\documentclass[11pt]{amsart}
\usepackage{amsmath, amsrefs, amsthm, amssymb}
\usepackage[active]{srcltx}		
\usepackage{pdfsync}					
\usepackage{graphicx}
\newcommand{\dd}{\,{\rm d}}

\newcommand\R{{\mathbb{R}}}

\newcommand\N{{\mathbb{N}}}
\newcommand\C{{\mathbb{C}}}

\newcommand\Z{{\mathbb{Z}}}

\renewcommand\S{{\mathbb{S}}}
\newcommand\rodnl{\Bigl(\frac{3-\gamma}{2}u^2+\frac{\gamma}{2}u_x^2\Bigr)}

\DeclareMathOperator\supp{\text{supp}}

\newtheorem{theorem}{Theorem}[section]
\newtheorem{proposition}[theorem]{Proposition}
\newtheorem{lemma}[theorem]{Lemma}
\newtheorem{corollary}[theorem]{Corollary}

\theoremstyle{definition}

\theoremstyle{remark}
\newtheorem{remark}[theorem]{Remark}

\numberwithin{equation}{section}

\begin{document}

\title[Rod equation]
{On permanent and breaking waves in hyperelastic rods and rings}

\author{Lorenzo Brandolese and Manuel Fernando Cortez}

\address{Universit\'e de Lyon, Universit\'e Lyon 1,
CNRS UMR 5208 Institut Camille Jordan,
43 bd. du 11 novembre,
Villeurbanne Cedex F-69622, France.}
\email{brandolese{@}math.univ-lyon1.fr, cortez@math.univ-lyon1.fr}
\urladdr{http://math.univ-lyon1.fr/$\sim$brandolese}
\thanks{The authors are supported by the French ANR Project DYFICOLTI. The second author is also supported by the 
{\it Secretar\'\i a Nacional de Educaci\'on Superior, Ciencia, Tecnolog\'\i a e Innovaci\'on\/}}


\begin{abstract}
We prove that  the only global strong solution of the periodic rod equation vanishing in at least one point $(t_0,x_0)\in \R^+\times\S^1$
is the identically zero solution. Such conclusion holds  provided the physical parameter~$\gamma$ of the model (related to the Finger deformation tensor) is outside some neighborhood of the origin and applies in particular
for the Camassa--Holm equation, corresponding to $\gamma=1$.
We also establish the analogue of this unique continuation result in the case of non-periodic solutions defined on the whole 
real line with vanishing boundary conditions at infinity.
Our analysis relies on the application of new local-in-space blowup criteria and involves the computation of several best constants
in convolution estimates and weighted Poincar\'e inequalities.
\end{abstract}

\maketitle

\begin{center}
{The original publication is available 
online in:\\
J. Funct. Anal. (2014), {\tt http://dx.doi.org/10.1016/j.jfa.2014.02.039}}
\end{center}
\medskip

 \section{Introduction}
\label{sec:intro}
 
 \subsection*{Motivations}
This paper is devoted to the study of periodic solutions of the rod equation.
A first  physical motivation comes from the study of the response of hyper-elastic rings under the action of an initial radial stretch.
As the nonlinear dispersive waves propagating inside it could eventually lead to cracks, an important problem is the determination of conditions that must be fulfilled in order to prevent their formation. The main issue of the present paper will be  a precise description of crack mechanisms inside  such rings.

A second reason for studying periodic solutions is that periodic waves spontaneously arise also in  hyper-elastic rods: indeed, it has been recently observed that the solitary waves propagating inside an ideally infinite length rod can suddenly feature a transition into waves with finite period as their amplitude increases, see \cite{Dai-Huo}.

Our third motivation comes from the study of shallow water waves inside channels.
Indeed, the Camassa--Holm equation (at least in the dispersionless case) is a particular case,
corresponding to $\gamma=1$, of the rod equation below:
if the motion of small amplitude waves is usually modeled by the KdV equation, larger amplitude waves, and in particular breaking waves, are more accurately  described by the Camassa--Holm equation. 
In fact, both the KdV and  the Camassa--Holm equation 
can be rigorously derived as an asymptotic model from the free surface Euler equations for irrotational inviscid flows,
in the so-called shallow water regime  $\mu=h^2/\lambda^2<\!\!\!<1$, where $h$ and $\lambda$ denote respectively the average elevation of the liquid over the bottom and the characteristic wavelength.
The Camassa--Holm equation models small, but finite, amplitude waves, {\it i.e.\/} waves such that the dimensionless amplitude parameter  $\epsilon=a/h$ satisfies $\epsilon=O(\sqrt\mu)$, where $a$ is the typical amplitude, whereas the derivation of KdV would require the more stringent scaling $\epsilon=O(\mu)$.
The Camassa--Holm equation thus better captures the genuinely nonlinear behavior of larger amplitude waves and, contrary to KdV,  admits both permanent solutions and  solutions that blow up in finite time. For a detailed discussion on these issues, see \cites{CamHolHym94, ConEschActa, AConL09, John02}.

 Let $\S=\R/\Z$ be the unit circle. The Cauchy problem for the periodic rod equation is written as follows:
\begin{equation}
\label{rod}
\begin{cases}
u_t+\gamma uu_x=-\partial_x p*\biggl(\displaystyle\frac{3-\gamma}{2}\,u^2+\frac{\gamma}{2}\,u_x^2\biggr), &\;t\in(0,T), \; x\in\S,\\
u(0,x)=u_0(x).
\end{cases}
\end{equation}
The real parameter $\gamma$ is related to the Finger deformation tensor of the material. Both positive and negative values~$\gamma$ are admissible.
We refer to \cites{Dai-Huo} for more details on the physical background and the mathematical derivation of the model.

The function~$p$ in~\eqref{rod} is the kernel of the convolution operator $(1-\partial^2_x)^{-1}$.
It is the continuous $1$-periodic function given by 
  \begin{equation}
  \label{kernel}
  p(x)=\frac{\cosh(x-[x]-1/2)}{2\sinh(1/2)},
 \end{equation}
where $[\cdot]$ denotes the integer part.

The Camassa--Holm case  $\gamma=1$ is somewhat particular, as in this case~\eqref{rod} inherits a bi-Hamiltonian structure and the equation is completely integrable,  in the sense of infinite-dimensional Hamiltonian systems: in suitable variables (action-angle), 
the flow is equivalent to a linear flow at constant speed on the Jacobi variety 
associated to a (mostly infinite-dimensional) torus, cf.  the discussion in \cite{ConMc}.
Moreover,  the Camassa-Holm equation is a re-expression of geodesic flow on  the diffeomorphism group of the circle, see
the discussion in \cites{ConKol03, Kol07}.
For these reasons, many important results valid for the Camassa--Holm equation do not go through in the general case. 
For example, in the case of the Camassa--Holm equation on the real line, a striking necessary and sufficient condition for the global existence of strong solution can be given in terms on the  initial potential $y_0=u_0-u_{0,xx}$, see~\cite{McKean04}. 
On the other hand, very little is known on the global existence of strong solutions when $\gamma\not=1$: it can even happen (when $\gamma=3$) that all nonzero solutions blow up in finite time.
Smooth solitary waves that are global strong solutions 
were constructed at least for some~$\gamma$, see~\cites{Dai-Huo, Lenells-DCDS06}. 
These are essentially the only known examples of global smooth solutions.

Our working assumption will be that $u_0$ belongs to the Sobolev space $H^s(\S)$, for some $s>3/2$.
Then, for any $\gamma\in\R$, the Cauchy problem for the rod equation is locally well-posed, in the sense that 
there exist a maximal time $0<T^*\le\infty$
and a unique solution $u\in C([0,T^*),H^s(\S))\cap C^1([0,T^*),H^{s-1}(\S))$.
Moreover, the solution~$u$ depends continuously on the initial data.
It is also known that~$u$ admits several invariant integrals, among which the energy integral,
\[
E(u)=\int_\S (u^2+u^2_x)\dd x.
\]
In particular, because of the 
conservation of the Sobolev $H^1$-norm, the solution $u(x,t)$ remains uniformly bounded up to the time~$T^*$.
On the other hand, if $T^*<\infty$ then $\limsup_{t\to T}\|u(t)\|_{H^s}=\infty$ ($s>3/2$) and the precise blowup scenario,
often named wave breaking mechanism, is the following:
\begin{equation}
\label{blowup-scena}
T^*<\infty \iff \liminf_{t\to T^*}\Bigl(\inf_{x\in\S}\gamma u_x(t,x)\Bigr)=-\infty.
\end{equation}
See, {\it e.g.\/} \cite{ConStra00}.

\subsection*{Quick overview of the main results}
We will state our two main theorems in the next section after preparing some notations.
Loosely, the first theorem asserts that if $|\gamma|$ is \emph{not too small}, then there exist a constant 
$\beta_\gamma>0$ such that if
\begin{equation}
\label{asspi}
u'_0(x_0)>\beta_\gamma |u_0(x_0)| \quad \text{if  $\gamma<0$}, \qquad \text{or}\qquad
u'_0(x_0)<-\beta_\gamma |u_0(x_0)|\quad  \text{if $\gamma>0$}
\end{equation}
in at least one point $x_0\in\S$, then the solution arising from $u_0\in H^s(\S)$ must blow up in finite time.

Our second theorem quantifies the above result: it precises what ``$|\gamma|$ not too small'' means, addressing also the delicate issue of finding sharp estimates for~$\beta_\gamma$.
For example, in the particular case of the periodic {\it Camassa--Holm\/} equation we get that a sufficient condition for the blowup is:
\begin{equation}
\label{buch}
\exists\, x_0\in\S\colon\quad u_0'(x_0)<
   -\sqrt{\frac{5}{2}
  - \frac32\cdot\frac{\cosh\frac12\cosh\frac32-1 }{\sinh\frac12\sinh\frac32}\,} \,|u_0(x_0)|.
\end{equation}

An analogue but weaker result was recently established in our previous paper~\cite{BraCMP}, where we dealt with non-periodic
solutions on the whole real line with vanishing boundary conditions as $x\to\infty$.
In the present paper, we take advantage of the specific structure of periodic solutions in order to make improvements on our previous work in two directions.

First of all, the analogue blowup result for the rod equation on~$\R$ could be established in~\cite{BraCMP} only in the range $1\le\gamma\le4$ in the non-periodic case. 
But the relevant estimates on the circle that we will establish turn out to be much stronger.
They allow us to cover blowup results of periodic solutions {\it e.g.\/} for arbitrary large $\gamma$ (or $\gamma$ negative and arbitrary small).
The following corollary of  Theorem~\ref{maintheo} is therefore specific to periodic solutions:
\begin{corollary}
\label{coro-as}
There exists an absolute constant $\beta_\infty$ (it can be checked numerically that $\beta_\infty=0.295\ldots$) with the following property. 
If $u_0\in H^s(\S)$, with $s>3/2$, is such that for some $x_0\in\S$, $u'_0(x_0)>\beta_\infty |u_0(x_0)|$,
 or otherwise 
$ u'_0(x_0)<-\beta_\infty|u_0(x_0)|$,
then the solutions (unique, but depending on~$\gamma$) of the rod equation~\eqref{rod}  arising from 
$u_0$  blow up in finite time respectively if $\gamma<\!\!\!<-1$ or $\gamma>\!\!\!>1$. In both cases, the maximal existence time is $T^*=O(\frac{1}{|\gamma|})$ as $\gamma\to\infty$.
\end{corollary}

There is a  second important difference between  in the behavior of periodic and non-periodic solutions.
It can happen that two initial data $u_0\in H^s(\S)$ and $\widetilde u_0\in H^s(\R)$ agree on an arbitrarily large finite interval,
and that periodic solution arising from $u_0$ blows up, whereas the solution arising from $\widetilde u_0$ and vanishing at infinity
exists globally.
As a comparison, the blowup criterion in~\cite{BraCMP} for solutions in $H^s(\R)$ of the Camassa--Holm equation reads
$\inf_{\R}(u_0'+|u_0|)<0$; on the other hand, according to~\eqref{buch}, in the periodic case the condition 
$\inf_{\S}(u_0'+0.515\, |u_0|)<0$ would be enough for the development of a singularity.
In general, for $\gamma\in[1,4]$, the coefficient $\beta_\gamma$ in~\eqref{asspi} is considerably lower than
the corresponding coefficient $\beta_{\gamma,\R}$ computed in~\cite{BraCMP} for the blowup criterion of non-periodic solutions.

The most important feature of our blowup criteria~\eqref{asspi}-\eqref{buch} is that it they are 
 \emph{local-in-space\/}.
This means that these criteria involve a condition {\it only on a small neighborhood of a single point\/} of the datum. 
Their validity is somewhat surprising, as equation~\eqref{rod} is non-local. A huge number of previous papers addressed the blowup issue of solutions to equation~\eqref{rod},
(see, {\it e.g.\/} \cites{CamHol93, CamHolHym94, ConstAIF00, ConEschPisa, ConEschActa, Guo-Zhou-SIAM, Hu-Yin10,
Jin-Liu-Zhou10, Liu-MathAnn06, Wah 06,Wah-NoDEA07, Zhou2004, Zhou-Math-Nachr, Zhou-Calc-Var}; (the older references only dealt with the Camassa--Holm equation).
 But the corresponding blowup criteria systematically involved the computation of some global quantities of $u_0$: typically,  conditions of the form $u_0'(x_0)<-c_\gamma\|u_0\|_{H^1(\S^1)}$ or some other integral conditions on $u_0$, or otherwise antisymmetry conditions, etc.

The main idea (inspired to us combining those of \cites{BraCMP, ConstAIF00, McKean98}) will be to study the evolution of $u+\beta u_x$ and $u-\beta u_x$ along the trajectories of the flow map of~$\gamma u\,$, for an appropriate parameter~$\beta=\beta(\gamma)$ to be determined in order to optimize the blowup result.
The main technical issue of the present paper will be the study of a two-parameters family of minimization problems. 
Such minimization problems arise when  computing the best constants  in the relevant convolution estimates.
Calculus of variation tools have been already successfully used for the study of blowup criteria of the rod equation, see, {\it e.g.\/}, \cites{Wah 06,Wah-NoDEA07, Zhou2004, Zhou-Math-Nachr, Zhou-Calc-Var}. The difference with respect to these papers is that our approach requires the minimization of non-coercive functionals.
Once such minima are computed, the proof of the blowup can follow the steps of~\cite{BraCMP}.

We finish this introduction by mentioning a second simple consequence of our main Theorem~\ref{maintheo}, that seems to be of independent interest.
Its result  applies in particular to the case $\gamma=1$. It is a new result (at best of our knowledge)  also for the Camassa--Holm equation.

\begin{corollary}
 \label{cnge}
 Let $u\in C([0,\infty[, H^s(\S))\cap C^1([0,\infty[,H^{s-1}(\S))$, be a global solution of the rod equation~\eqref{rod} with  $\gamma\le\gamma_1^-$ or $\gamma \ge \gamma_1^+$ (where $\gamma_1^-=-1.036\ldots$ and $\gamma_1^+= 0.269\ldots$).
 If $u$ vanishes at some point $(t_0,x_0)\in[0,\infty)\times \S$, then $u$ must be the trivial solution: $u(x,t)\equiv0$ for all $t\ge0$ and $x\in\S$.
\end{corollary}

Corollary~\ref{cnge} improves an earlier result by A.~Constantin and J.~Escher \cite{ConEschCPAM},
asserting that the trivial solution is the only global solution of the periodic Camassa--Holm equation such that \emph{for all $t\ge0$, $\exists\, x_t\in\S$ such that $u(t,x_t)=0$}.
Basically, we manage to replace their condition~``$\forall\, t$\dots'' by the much weaker one ``$\exists\,t$\dots''.
More importantly,  unlike~\cite{ConEschCPAM}, our approach will make use of only few properties of the equation, and is more suitable for generalizations.

It might be surprising that our results {\it a priori\/} excludes a small neighborhood of the origin for the parameter~$\gamma$. In fact, such restriction might be purely technical. However, one should observe that
 $\gamma=0$ \emph{must be excluded}. The reason is that for $\gamma=0$ all solutions arising from $u_0\in H^s(\S)$ are global in time. Indeed,  the blowup scenario~\eqref{blowup-scena} is never fulfilled.
It is worth noticing that for  $\gamma=0$  the rod equation reduces to the BBM equation  introduced by Benjamin, Bona and Mahony in~\cite{BBM72} as a model for surface wave in channels.

We will establish a result similar to that of Corollary~\ref{cnge} for non-periodic solutions vanishing for $x\to\infty$ (see Corollary~\ref{cnge2} at the end of the paper for the precise statement).
In this case, we will need that the global solution~$u$ is such that 
$u(t_0,x_1)=u(t_0,x_2)=0$ in at least two different points, or otherwise that $u(t_0,\cdot)$ decays
sufficiently fast as $x\to\infty$ to conclude that $u(t,x)\equiv0$ for all $(t,x)\in \R^+\times \R$.

\subsection*{Organization of the paper.}
In Section~\ref{sec:two} we state our main results, Theorem~\ref{maintheo} and Theorem~\ref{theoquant}.
The first theorem is proved  in Section~\ref{sec:first p}. In Section~\ref{sec:fo} we study a minimization problem that will play an important role for the proof  of Theorem~\ref{theoquant}, given in Section~\ref{sec:tq}.
Corollary~\ref{coro-as} will immediately follows from assertion~(ii) of this theorem. 
We will study in more detail the Camassa--Holm equation in Section~\ref{sec:ch}.
Corollary~\ref{cnge} is established in Section~\ref{sec:cor} where we also provide some new results for non-periodic solutions.
An appendix showing the agreement between our theorems and some numerical results concludes the paper.

\section{The main results}
\label{sec:two}

We start preparing some notations.
For any real constant $\alpha$ and $\beta$, let $I(\alpha,\beta)\ge -\infty$ defined by
\begin{equation}
\label{iab}
I(\alpha,\beta)=
\inf\biggl\{ \int_0^1 \bigl(p+\beta p_x)\Bigl(\alpha u^2+ u_x^2\Bigr)\dd x \colon u\in H^1(0,1),
\; u(0)=u(1)=1\biggr\}.
\end{equation}
Proposition~\ref{prop1} below will characterize the set of parameters $(\alpha,\beta)$ for which
the functional appearing in~\eqref{iab} is bounded from below as well as the subset for which the infimum is achieved.

We also introduce, for $\gamma\in\R^*$, the quantity $\beta_\gamma\in[0,+\infty]$ defined by
\begin{equation}
\label{def:beta}
\beta_\gamma=\inf\Bigl\{ \beta\in\R^+\colon 
\beta^2+ I\bigl(\textstyle\frac{3-\gamma}{\gamma},\beta\bigr)-\textstyle\frac{3-\gamma}{\gamma}\ge0\Bigr\},
\end{equation}
with the usual convention that $\beta_\gamma=+\infty$ if the infimum is taken on the empty set.

Our main result is the following blowup theorem:
\begin{theorem}
\label{maintheo}
Let $\gamma\in\R^*$ be such that 
$\beta_\gamma<+\infty$.
Let $u_0\in H^s(\S)$ with $s>3/2$ and assume that there exists $x_0\in\S$, such that
\begin{equation}
\label{assp}
u'_0(x_0)>\beta_\gamma |u_0(x_0)| \quad \text{if  $\gamma<0$}, \qquad or\qquad
u'_0(x_0)<-\beta_\gamma |u_0(x_0)|\quad  \text{if $\gamma>0$}.
\end{equation}
Then the corresponding solution  $u$ of equation~\eqref{rod} in $C([0,T^*),H^s(\S))\cap C^1([0,T^*),H^{s-1}(\S))$  arising from $u_0$ blows up in finite time. Moreover, the maximal time~$T^*$ is estimated by
\begin{equation}
\label{tmax}
T^*\le\frac{2}{\gamma\sqrt{u_0'(x_0)^2-\beta_\gamma^2 u_0(x_0)^2}\,}
\end{equation}
and, for some $x(t)\in\S$, the blowup rate is 
\begin{equation} \label{burate} 
 u_x(t,x(t))\sim -\frac{2}{\gamma(T^*-t)} \qquad\text{as $t\to T^*$}.
\end{equation}
\end{theorem}

Theorem~\ref{maintheo} is meaningful only if $\gamma$ is such that $\beta_\gamma<+\infty$.
The validity of such condition is {\it a priori\/} not clear, as it might happen that the set in equation~\eqref{def:beta} 
is empty. 
 An important part of the present paper will be devoted to the technical issue of discussing the validity of the
 condition $\beta_\gamma<+\infty$ 
and next estimating $\beta_\gamma$.
This will be done by establishing some sharp estimates on $I(\alpha,\beta)$.

In order to state our next theorem, let us introduce the complex number
\[
 \mu=\frac12\sqrt{1+4(3-\gamma)/\gamma}, \qquad \gamma\not=0,
\]
where $\sqrt{1+4(3-\gamma)/\gamma}$ denotes any of the two complex square roots.
We also consider the four constants (that will be constructed in~\eqref{4ga} below):
\begin{equation*}
\begin{matrix}
 \gamma_1^-=-1.036\ldots \quad
 &\gamma_1^+=0.269\ldots\\
 \gamma_2^-=-1.508\dots\quad
& \gamma_2^+=0.575\dots
\end{matrix}
\end{equation*}

\begin{theorem}
\mbox{}
\label{theoquant}
\begin{itemize}
\item[(i)]
For any $\gamma\in (-\infty,\gamma_1^-]\cup[\gamma_1^+,+\infty)$, we have
$\beta_\gamma<+\infty$, so that Theorem~\ref{maintheo} applies in such range.
More precisely, if $\gamma\in(-\infty,\gamma_2^-]\cup[\gamma_2^+,\infty)$, then
\begin{equation}
 \label{bobi}
 \beta_\gamma\le\sqrt{\frac{3}{\gamma}-\frac12
  - \mu\cdot\frac{\cosh\frac12\cosh\mu-1 }{\sinh\frac12\sinh\mu}}.
\end{equation}
If otherwise  $\gamma\in[\gamma_2^-,\gamma_1^-]\cup[\gamma_1^+,\gamma_2^+]$,
then $\beta_\gamma$ can be estimated as in inequality~\eqref{est:betb} below (in~\eqref{est:betb}, $P_\nu$ denotes Legendre function of the first kind of degree~$\nu$).
\item[(ii)]
The limit $\beta_\infty=\lim_{\gamma\to\pm\infty}\beta_\gamma$ does exist and 
\begin{equation}
 \label{assii}
\beta_\infty\le 
 \sqrt{
 \frac{\sqrt3\bigl(1-\cosh\frac12\cos\frac{\sqrt3}{2}\bigr)}{2\sinh\frac12\sin\frac{\sqrt3}{2}\,}  -\frac12   }
 \,=0.296\dots.
\end{equation}
\end{itemize}
\end{theorem}

It is worth observing that we have $\beta_\gamma=0$ if and only if $\gamma=3$. This can be checked directly from the definitions~\eqref{iab}-\eqref{def:beta}. Accordingly, the right-hand side in~\eqref{bobi} vanishes for 
$\gamma=3$. In particular, we recover the known fact (see \cite{ConStra00}) that for $\gamma=3$ any nonzero initial datum $u_0\in H^s(\S)$ (with $s>3/2$) gives rise to a solution that blows up in finite time. Moreover, the maximal time~$T^*$ satisfies
\[T^*\le \frac{2}{3}\sqrt{-\inf_{x\in\S} u_0'(x)}\,.\]

%

The upper bounds~\eqref{bobi}--\eqref{assii} are not optimal but, strikingly, they are {\it almost sharp\/}.
Indeed we can compute the numerical approximation of $\beta_\gamma$ with arbitrary high precision.
We find that the error between the above bounds and the numerical value is only of order $10^{-3}$.
For example, we find $\beta_\infty=0.295\dots$ that is indeed very close to the bound~\eqref{assii}.
Moreover, in the Camassa--Holm case, we find numerically that $\beta_1=0.513\ldots$. This is in good agreement with estimate~\eqref{bobi}, that for $\gamma=1$ provided us with the bound $\beta_1\le 0.515\dots$.

We will devote the appendix to a more detailed discussion of the numerical results. 
Such analysis will also show that the effective range of applicability of Theorem~\ref{maintheo}
is slightly larger than the range  $\gamma\in(-\infty,\gamma_1^-]\cup[\gamma_1^+,\infty)$ predicted by
Theorem~\ref{theoquant}. The reader should compare Figure~\ref{upbound-betag} with Figure~\ref{courbe-beta-gamma} at the end of the paper: the former plot is obtained analytically and the latter numerically.

\begin{figure}
\includegraphics[width=15cm,height=9cm]{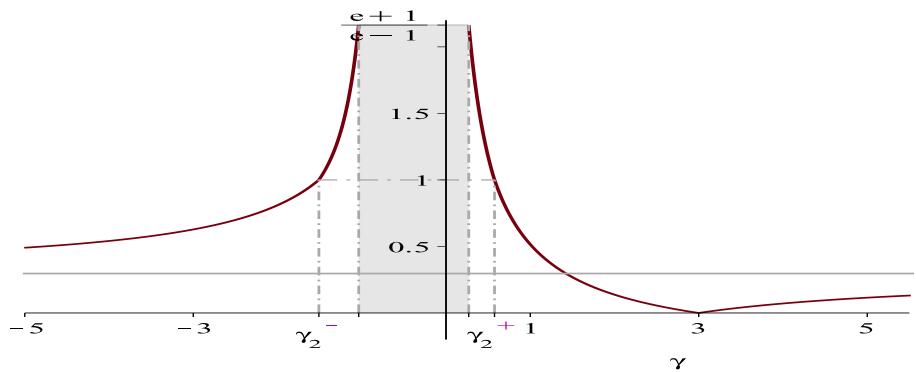}{\centering}
\vskip-3.5cm
\caption{The upper-bound estimate of~$\beta_\gamma$ given by Theorem~\ref{theoquant}. The estimate is valid outside the interval $[\gamma_1^-,\gamma_1^+]$ 
(gray region).}
\label{upbound-betag}
\end{figure}

This paper does not address the problem of the continuation of the solution after the blowup as
this issue is extensively studied in the literature. See, {\it e.g.\/}  \cites{BreCon07, BreCon07bis, HolJDE07, Mus07}.

\section{First properties of~$I(\alpha,\beta)$ and proof of Theorem~\ref{maintheo}}
\label{sec:first p}

\bigskip
For any real~$\beta$, let us consider the 1-periodic function
\begin{equation} 
\label{32}
\omega(x)=p(x)+\beta p'(x),
\end{equation}
where $p$ is the kernel introduced in~\eqref{kernel} and $p'$ denotes the distributional derivative on~$\R$, that agrees in this case with the classical a.e. pointwise derivative on~$\R\backslash\Z$.

We would like  to make use of $\omega$ as a weight function. The non-negativity condition
$\omega\ge0$ is equivalent to the inequality $\cosh(1/2)\ge \pm\beta \sinh(1/2)$, {\it i.e.}, to the condition
\begin{equation} 
\label{bbeta}
 -\textstyle\frac{e+1}{e-1}\le \beta\le\textstyle \frac{e+1}{e-1}.
\end{equation}
Throughout this section, we will work under the above condition on~$\beta$.

Let us introduce the weighted Sobolev space
\begin{equation}
\label{def:Eb}
E_\beta
=\Bigl\{ u\in L^1_{\text{loc}}(0,1)\colon
 \|u\|_{E_\beta}^2\equiv\int_0^1\omega(x)( u^2+u_x^2)(x)\dd x<\infty\Bigr\},
\end{equation}
where  the derivative is understood in the distributional sense.
Notice that $E_\beta$ agrees with the classical Sobolev space $H^1(0,1)$ when $|\beta|<\frac{e+1}{e-1}$,
as in this case $\omega$ is bounded and bounded away from~$0$, and the two norms  $\|\cdot\|_{E_\beta}$ and $\|\cdot\|_{H^1}$
are equivalent: in particular, in this case if $u\in E_\beta$ then $u=\tilde u$ a.e., where $\tilde u\in C([0,1])$.
The situation is different for $\beta=\pm\frac{e+1}{e-1}$ as $E_\beta$ is strictly larger than $H^1(0,1)$ in this case. Indeed, we have 
\begin{equation}
\label{def:omegal}
\omega(x)=\frac{2e}{(e-1)^2}\sinh(x), \qquad x\in(0,1), \qquad (\text{if $\beta=\textstyle\frac{e+1}{e-1}$}).
\end{equation}
An element~$u$ of~$E_{(e+1)/(e-1)}$, after modification on a set of measure zero, agrees with a function $\tilde u$ that is
continuous on $(0,1]$, but may be unbounded for $x\to0^+$ (for instance, $|\log (x/2)|^{1/3}\in E_{(e+1)/(e-1)}$).
In the same way,
\begin{equation}
\label{def:omegal2}
\omega(x)=\frac{2e}{(e-1)^2}\sinh(1-x), \qquad x\in(0,1), \qquad (\text{if $\beta=\textstyle-\frac{e+1}{e-1}$}).
\end{equation}
After modification on a set of measure zero, the elements of~$E_{-(e+1)/(e-1)}$ are continuous on $[0,1)$,
 but may be unbounded for $x\to1^-$.

Let us consider the closed subspace $E_{\beta,0}$ of~$E_\beta$ defined
as the closure of $C_c^\infty(0,1)$ in~$E_\beta$.
Notice that, with slightly abusive notation, consisting in identifying $u$ with its continuous representative $\tilde u$ we have:
\begin{equation}
 \label{ebo}
 \begin{split}
 &E_{\beta,0}=H^1_0(0,1)=\{u\in H^1(0,1)\colon u(0)=u(1)=0\},\qquad\text{if $|\beta|<\textstyle\frac{e+1}{e-1}$}.\\
 \end{split}
 \end{equation}
On the other hand, in the limit cases for~$\beta$ we have the following: 
\begin{lemma}\mbox{}
 \label{lem:bigob}
 \begin{itemize}
  \item[-]
 If $u\in E_{(e+1)/(e-1),0}$, then $u(x)=O(\sqrt{|\log x|})$ as $x\to0^+$ and $u(1)=0$.
 \item[-]
  If $u\in E_{-(e+1)/(e-1),0}$, then $u(x)=O(\sqrt{|\log(1-x)|})$ as $x\to1^-$ and $u(0)=0$.
 \end{itemize}
\end{lemma}

\begin{proof}
 We consider only the case $\beta=\frac{e+1}{e-1}$ as in the other case the proof is similar.
 The condition $u(1)=0$ follows from the fact that $\omega$ is bounded and bounded away from the origin in a left 
 neighborhood of~$1$.
 Moreover.
 \[
\begin{split}
|u(x)|=\left|\int_x^1 u'(y)\dd y\right|
&\le 
\biggl(\int_0^1 \omega(y) u'(y)^2\dd y\biggr)^{1/2}\biggl(\int_x^1 \frac{1}{\omega(y)}\dd y\biggr)^{1/2}\\
&\le C\|u\|_{E_\beta}\sqrt{|\log x|}.
\end{split}
\]
\end{proof}

%
%
%
%

The elements of $E_{\beta,0}$ satisfy to the weighted Poincar\'e inequality below:
\begin{lemma}
\label{lem:poi}
For all $-\frac{e+1}{e-1}\le \beta\le \frac{e+1}{e-1}$, there exists a constant $C>0$ such that 
\begin{equation}
\label{poincare}
\forall\, v\in E_{\beta,0} \colon 
\int_0^1\omega(x)v(x)^2\dd x \le  C \int_0^1\omega(x)v_x(x)^2\dd x.
\end{equation}
\end{lemma}

\begin{proof}
The validity of such inequality is obvious for $|\beta|<\frac{e+1}{e-1}$. Indeed, in this case there exist two constants
$c_1$ and $c_2$ such that on the interval $(0,1)$ we have
$0<c_1\le \omega(x)\le c_2$ and the validity of~\eqref{poincare} is reduced to that of the classical Poincar\'e inequality without weight.

In the limit case $\beta=\frac{e+1}{e-1}$, we can observe that, from~\eqref{def:omegal},
the only zero of the function $\omega(x)=\frac{2e}{(e-1)^2}\sinh(x)$ in the closure of $(0,1)$ is of order~one.
Then the weight $\omega(x)$ satisfies the necessary and sufficient condition for the weighted Poincar\'e inequality to hold, see  \cite{Stre84}.
For reader's convenience we prove directly inequality~\eqref{poincare} exhibiting  an explicit constant.
Recall that $v(x)=O(|\log(x)|^{1/2})$ as $x\to0^+$.
In particular, for all $v\in E_{(e+1)/(e-1),0}$,  we have $\bigl[ \omega(x)v(x)^2\bigr]_{0+}^{1-}=0$.
Then integrating by parts and next using the Cauchy-Schwarz inequality we obtain, for $v\in E_{(e+1)/(e-1),0}$\,:
\[
\frac{1}{4}\Bigl(\int_0^1 \cosh v^2\Bigr)^2=\Bigl(\int_0^1 \sinh v v_x\Bigr)^2
\le  \Bigl(\int_0^1\sinh v^2\Bigr)\Bigl(\int_0^1\sinh v_x^2\Bigr).
\]
Next observe that
\[
\forall\, x\in(0,1), \qquad \cosh(x)\ge \frac{e^2+1}{e^2-1}\sinh(x).
\]
Combining these two estimates, we obtain
\begin{equation}
\label{poincas}
\int_0^1 \sinh(x) v(x)^2\dd x\le \frac{4(e^2-1)^2}{(e^2+1)^2}\int_0^1 \sinh(x) v_x(x)^2\dd x.
\end{equation}
Then~\eqref{poincare}~holds, {\it e.g.\/}, with $C=\frac{4(e^2-1)^2}{(e^2+1)^2}$.
In the other limit case  $\beta=-\frac{e+1}{e-1}$, $\omega(x)=\frac{2e}{(e-1)^2}\sinh(1-x)$.
Therefore, we can reduce to the previous case with a change of variables.
\end{proof}
The constant in ~\eqref{poincas} is far from being optimal. 
In fact, we will find the best constant in Remark~\ref{rem:meco},
 as a byproduct of the analysis performed Section~\ref{sec:fo}.

This being observed, let us go back to our minimization problem
\begin{equation}
\label{iab2}
I(\alpha,\beta)=
\inf\biggl\{ \int_0^1 \bigl( p+\beta p')\Bigl(\alpha u^2+ u_x^2\Bigr)\dd x \colon u\in H^1(0,1),\; u(0)=u(1)=1
\biggr\}.
\end{equation}

\begin{proposition}
\mbox{}
\label{prop1}
We have
\[
I(\alpha,\beta)>-\infty \iff 
\begin{cases} 
 -\frac{e+1}{e-1}\le \beta\le \frac{e+1}{e-1},\\ 
\alpha>-1/{C(\beta)},
\end{cases}
\]
where $C(\beta)>0$ is the \emph{best constant} in the  weighted Poincar\'e inequality~\eqref{poincare}.
Moreover, if $|\beta|< \frac{e+1}{e-1}$, then $I(\alpha,\beta)$ is in fact a minimum and there is only one minimizer
$u\in H^1(0,1)$ with $u(0)=u(1)=1$.
\end{proposition}

\begin{proof}
Putting $u=v+1$ and observing that $\int_0^1\omega(x)\dd x=1$, we see that
\begin{equation}
\label{ij}
I(\alpha,\beta)=\alpha+\inf\bigl\{J(v)\colon v\in H^1_0(0,1)\bigr\},
\end{equation}
where
\begin{equation}
\label{JJ}
J(v)=\int_0^1 \omega(x)(\alpha v^2+v_x^2+2\alpha v)(x)\dd x.
\end{equation}

Assume that $I(\alpha,\beta)>-\infty$.
Then $|\beta|\le \frac{e+1}{e-1}$, otherwise we would get a contradiction by taking
a sequence of the form $\chi(x)\cos(nx)$ with $\chi$ smooth and such that 
$\supp(\chi)\subset\supp(\omega^-)\cap(0,1)$, where $\omega^-$ denotes the negative part of~$\omega$.
To prove the second inequality $\alpha>-1/C(\beta)$, we only have to treat the case $\alpha<0$.
Applying the inequality
\[
\int_0^1\omega(\alpha n^2v^2+n^2v_x^2+2\alpha nv)\ge I(\alpha,\beta)-\alpha
\]
valid for all $v\in H^1_0(0,1)$ and all $n\in\N$ and letting $n\to\infty$, we get
\[
\int_0^1 \omega v^2\le -\frac{1}{\alpha}\int_0^1 \omega v_x^2.
\]
Then we get $\alpha\ge -1/C(\beta)$. But in fact the inequality is strict, as otherwise we could take
a sequence $(v_n)$ such that $(-\alpha\int_0^1 \omega v_n^2)/(\int \omega v_{x,n}^2)\to1$ and
$\int \alpha\omega v_n\to-\infty$ to get a contradiction.

Conversely, assume  that  $|\beta|\le\frac{e+1}{e-1}$.
By the weighted Poincar\'e inequality~\eqref{poincare},
the map $v\mapsto\int_0^1 \omega v_x^2$ defines on $E_{\beta,0}$ an equivalent norm.
As $\alpha<-1/C(\beta)$, the symmetric bilinear form $B(v_1,v_2)=\int_0^1 \omega(\alpha v_1v_2+v_1'v_2')$ is coercive on the Hilbert space~$E_{\beta,0}$. 
Applying the Lax-Milgram theorem yields the existence and the uniqueness of a minimizer $\bar v\in E_{\beta,0}$ 
for the functional $J$.
But $H^1_0(0,1)\subset E_{\beta,0}$, so in particular, we get $I(\alpha,\beta)> -\infty$.
Moreover, if $|\beta|<\frac{e+1}{e-1}$, then recalling $E_{\beta,0}=H^1_0$ we see that $I(\alpha,\beta)$ is in fact a minimum, achieved at $\bar u=1+\bar v\in H^1$.
\end{proof}

The next proposition provides some useful information on~$I(\alpha,\beta)$. 
\begin{proposition}
\label{prop2}
The function $(\alpha,\beta)\mapsto I(\alpha,\beta)\in \R\cup\{-\infty\}$, defined for all $(\alpha,\beta)\in\R^2$,  is concave with respect to each one of its variables and  is even with respect to the variable~$\beta$.
Moreover,
\begin{equation}
\label{maxi}
\text{$\forall\, \alpha\in\R,\quad\forall\,|\beta|\le \textstyle\frac{e+1}{e-1}$},\qquad  
-\infty\le I(\alpha,\textstyle\frac{e+1}{e-1})\le I(\alpha,\beta)\le I(\alpha,0)\le \alpha.
\end{equation}
\end{proposition}

\begin{figure}
\includegraphics[width=17cm,height=11cm]{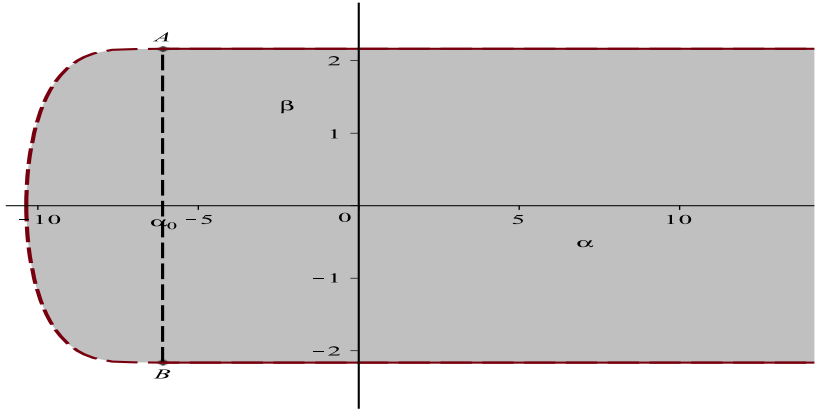}{\centering}
\vskip-5cm
\caption{The set of points $(\alpha,\beta)$ such that $I(\alpha,\beta)>-\infty$ (gray region).
The infimum is achieved in $H^1_0$ on its interior and in the larger space $E_{\beta,0}$ on the upper and lower boundaries $\beta=\pm\frac{e+1}{e-1}$.
The  abscissa $\alpha_0=-1/C(\mp\frac{e+1}{e-1})$ of the vertical dashed segment is the largest zero of a suitable Legendre function (see~Subsection~\ref{ssec:leg}).}
\label{ifini}
\end{figure}

\begin{proof}
The concavity property follows from the fact that $I(\alpha,\beta)$ is defined as an infimum of affine functions of the variables $\alpha$ and $\beta$. 
 To prove that $I(\alpha,\beta)=I(\alpha,-\beta)$, we can observe that 
 \[(p+\beta p')(x)=(p-\beta p')(1-x)\]
 and conclude making the change of variable $y=1-x$ inside the integral in~\eqref{iab2}.
 
To prove the last inequality in~\eqref{maxi}, consider the constant function $u\equiv1$ and 
observe that $I(\alpha,0)\le \alpha \int_0^1 (p+\beta p')(x)\dd x=\alpha$.
The other inequalities follow from the concavity and parity properties of the map $\beta\mapsto I(\alpha,\beta)$.
The result of Proposition~\ref{prop1} is also needed for the second inequality.

\end{proof}

Next lemma motivates the introduction of the quantity $I(\alpha,\beta)$ in relation with the rod equation.

\begin{lemma}
\label{lem:convo}
For any $\alpha,\beta\in\R$ and all $u\in H^1(\S)$ the following convolution estimate holds:
\begin{equation}
\label{convest}
\forall\,x\in\S,\qquad (p+\beta p')*\bigl(\alpha u^2+u_x^2)(x)\ge I(\alpha,\beta)\,u(x)^2
\end{equation}
and $I(\alpha,\beta)$ is the best possible constant.
\end{lemma}

\begin{remark}
This lemma is interesting when $I(\alpha,\beta)>-\infty$, but its statement is true also when
$I(\alpha,\beta)=-\infty$. In this case the sentence ``$-\infty$ is the best possible constant'' should be understood
as a negative result: the convolution estimate~\eqref{convest} breaks down whenever $(\alpha,\beta)$ does not fulfill
the conditions of Proposition~\ref{prop1}. 
\end{remark}

\begin{proof}
Let $\delta=\delta(\alpha,\beta)$ be some (possibly negative) constant.
Because of the invariance under translations,
we have that 
\begin{equation}
\label{bed}
(p+\beta p')*\bigl(\alpha u^2+u_x^2)\ge \delta\,u^2
\end{equation}
holds true for all $u\in H^1(\S)$ if and only if 
\[(p+\beta p')*\bigl(\alpha u^2+u_x^2)(1)\ge \delta\,u^2(1)\]
 holds true for all $u\in H^1(\S)$.
But on the interval~$(0,1)$,
$(p+\beta p')(1-x)=(p-\beta p')(x)$.
 Hence,
\[
(p+\beta p')*\bigl(\alpha u^2+u_x^2)(1)=\int_0^1 (p-\beta p')(x)(\alpha u^2+u_x^2)(x)\dd x.
\]
Normalizing to obtain $u(1)=1$ (and hence $u(0)=1$ by the periodicity) we get that the best constant 
$\delta$~in inequality~\eqref{bed} satisfies $\delta=I(\alpha,-\beta)=I(\alpha,\beta)$.
\end{proof}

\begin{proof}[Proof of Theorem~\ref{maintheo}.]
Let $s>3/2$. By the result of~\cite{ConStra00} we know that there exists  a unique solution~$u$ of the rod equation~\eqref{rod},  defined in some nontrivial interval $[0,T)$, and such that 
$u\in C([0,T),H^s(\S))\cap C^1([0,T),H^{s-1}(\S))$. Moreover, the map $u_0\mapsto u$ is continuous from $H^s(\S)$ to
$C([0,T),H^s(\S))\cap C^1([0,T),H^{s-1}(\S))$.
Owing to this well-posedness result, we can reduce to the case $s\ge3$. Indeed, if $u_0\in H^s(\S)$ with 
$3/2<s<3$, we can approximate $u_0$ in the $H^s(\S)$-norm using a sequence of data $(u_0)_n$ belonging to $H^3(\S)$ and satisfying condition~\eqref{assp}.
The relevant estimates, including that of~$T^*$, will pass to the limit as $n\to\infty$.

If, by contradiction, $T^*=+\infty$, then $T>0$ can be taken arbitrarily large.
As in~\cites{ConstAIF00, McKean98},
the starting point is the analysis of the flow map $q(t,x)$, defined by 
\begin{equation}
\label{flow}
\begin{cases}
q_t(t,x)=\gamma u(t,q(t,x)), &t\in(0,T),\; x\in\R,\\
q(0,x)=x.
\end{cases}
\end{equation}
For any $x\in\S$, the map $t\mapsto q(t,x)$ is well defined and continuously differentiable in the whole time interval $[0,T)$.
It is worth pointing out that for $\gamma=1$,~\eqref{flow} is the equation 
defining the geodesic curve of diffeomorphisms, issuing from the identity in the direction of 
$u_0$, cf. the discussion in~\cites{ConKol03, Kol07}. However, no such geometric 
interpretation is available for $\gamma \neq 1$.

From the rod equation
\begin{equation}
\label{rerod}
u_t+\gamma uu_x=-\partial_x p*\bigl(\textstyle\frac{3-\gamma}{2}\,u^2+\textstyle\frac{\gamma}{2}\,u_x^2\bigr),
\end{equation}
differentiating  with respect to the~$x$ variable and applying the identity $\partial^2_xp*f=p*f-f$,
we get
\begin{equation}
\label{eq:rodx}
\begin{split}
u_{tx}+\gamma uu_{xx}&=\frac{3-\gamma}{2}u^2-\frac{\gamma}{2}u_x^2-p*\rodnl\\
&=\frac{3}{2(\alpha+1)}\bigl[ \alpha u^2-u_x^2-p*(\alpha u^2+u_x^2)\bigr].
\end{split}
\end{equation}
Here we set
\begin{equation}
\label{def:alpha}
\alpha=\frac{3-\gamma}{\gamma}.
\end{equation}

Let us introduce the two~$C^1$-functions of the time variable, depending on~$\beta$,
\[ 
f(t)=(-u_x+\beta u)(t,q(t,x_0)), 
\qquad\text{and}
\qquad
g(t)=-(u_x+\beta u)(t,q(t,x_0)).
\]

Computing the time derivative using the definition of the flow~$q$, next using equations~\eqref{rerod}-\eqref{eq:rodx}, we get
\begin{equation}
\label{fpri}
\begin{split}
\frac{\dd f}{\dd t}(t)
&=\bigl[(-u_{tx}-\gamma u u_{xx})+\beta(u_t+\gamma uu_x)\bigr](t,q(t,x_0))\\
&=\frac{3}{2(\alpha+1)}\Bigl[-\alpha u^2+u_x^2+(p-\beta p')*(\alpha u^2+u_x^2)\Bigr](t,q(t,x_0)),
\end{split}
\end{equation}
and
\begin{equation}
\label{gpri}
\frac{\dd g}{\dd t}(t)= \frac{3}{2(\alpha+1)}\Bigl[-\alpha u^2+u_x^2+(p+\beta p')*(\alpha u^2+u_x^2)\Bigr](t,q(t,x_0)).
\end{equation}

Let us first consider the case $\gamma>0$. Then $\alpha>-1$.
From the definition of~$\beta_\gamma$~\eqref{def:beta} and the condition $\beta_\gamma<\infty$, we deduce that there exist $\beta\ge0$ such that
\begin{equation}
\label{betac}
\beta^2\ge \alpha-I(\alpha,\beta).
\end{equation}
Applying the convolution estimate~\eqref{convest} and recalling $I(\alpha,-\beta)=I(\alpha,\beta)$,
we get, for all $\beta\ge0$ satisfying~\eqref{betac},
\[
\begin{split}
\frac{\dd f}{\dd t}(t)&\ge \frac{3}{2(\alpha+1)}\Bigl[u_x^2-\bigl(\alpha-I(\alpha,-\beta)\bigr)u^2\Bigr](t,q(t,x_0))\\
&\ge \frac{3}{2(\alpha+1)}\Bigl[u_x^2-\beta^2 u^2\Bigr](t,q(t,x_0))\\
&=\frac{3}{2(\alpha+1)}f(t)g(t).
\end{split}
\]
In the same way,
\[
\begin{split}
\frac{\dd g}{\dd t}(t)&\ge \frac{3}{2(\alpha+1)}\Bigl[u_x^2-\bigl(\alpha-I(\alpha,\beta)\bigr)u^2\Bigr](t,q(t,x_0))\\
&\ge \frac{3}{2(\alpha+1)}\Bigl[u_x^2-\beta^2 u^2\Bigr](t,q(t,x_0))\\
&=\frac{3}{2(\alpha+1)}f(t)g(t).
\end{split}
\]

The assumption $u'_0(x_0)<-\beta_\gamma|u(x_0)|$ guarantees that we may choose
$\beta$ satisfying~\eqref{betac}, with $\beta-\beta_\gamma>0$ is small enough in a such way that
$u_0'(x_0)<-\beta |u_0(x_0)|$ For such a choice of~$\beta$ we have
\[ f(0)>0\quad \text{and}\quad g(0)>0.\]
The blowup of~$u$ will rely on the following basic property:

\begin{lemma}
\label{lemmadif}
Let $0<T^*\le\infty$ and  $f,g\in C^1([0,T^*),\R)$ be such that, for some constant~$c>0$ and 
all~$t\in [0,T^*)$,
\begin{equation}
\label{disyst}
\begin{split}
\displaystyle\frac{\dd f}{\dd t}(t)&\ge cf(t)g(t),\\ 
\displaystyle\frac{\dd g}{\dd t}(t)&\ge cf(t)g(t).
\end{split}
\end{equation}
If $f(0)>0$ and $g(0)>0$, then
\[
T^*\le \frac{1}{c\sqrt{f(0)g(0)}}<\infty.
\]
\end{lemma}

\begin{proof}
Let
\[
\tau=
\inf
\bigl\{ t\in[0,T^*)\colon f(t)=0\;\text{or}\;g(t)=0\bigr\}.
\]
The positivity of $f(0)$ and $g(0)$ implies that $\tau>0$. Observe that we cannot have $\tau<T^*$. 
Indeed, otherwise (exchanging if necessary $f$ with $g$) we would get $f(\tau)=0$ and $f,g\ge0$
on $[0,\tau]$. The differential inequality on~$f$ implies that $f$ is monotone increasing on $[0,\tau]$ leading to the contradiction $0=f(\tau)\ge f(0)>0$.
Hence, $f$ and $g$ are both monotone strictly increasing and positive in the whole interval $[0,T^*)$.

Now take $h(t)=\sqrt{f(t)g(t)}$. Using again assumption~\eqref{disyst}, next the arithmetic-geometric mean inequality
we find, on $[0,T^*)$,
\begin{align*}
\frac{\dd h}{\dd t}& \ge \frac{c}{2\sqrt{fg}} \,fg(f+g),\\
&\ge cfg=ch^2,
\end{align*}
with $h(0)=\sqrt{f(0)g(0)}$.
Substituting $v=1/f$, we find that 
$T^*\le\frac{1}{ c\sqrt{h(0)}}$.
\end{proof}

Estimate~\eqref{tmax} immediately follows in the case $\gamma>0$. When $\gamma<0$ the proof is the same, excepted for the fact that, as $\alpha<-1$, the inequalities must be reversed.

Let us establish the blowup rate~\eqref{burate}. More precisely, we prove that 
 \begin{equation*}
  \lim_{t\to (T^*)^-}(T^*-t)\,u_x\bigl(t,q(t,x_0)\bigr)=- \frac{2}{\gamma}.
  \end{equation*}
We will obtain such blowup rate adapting the arguments of~\cite{ConstAIF00}.
Namely, by Lemma~\ref{lemmadif} and the fact that, for some constant~$c>0$,  we have the estimate 
$\|u(t)\|_\infty\le c\|u_0\|_{H^1}$, we see that
$m(t)\equiv-\gamma u_x(t,q(t,x_0))\to+\infty$ as $t\to T^*$.
But $m$ satisfies the equation
\[
 \begin{split}
 \frac{\dd m}{\dd t}(t)
 &=\frac{3}{2(\alpha+1)}\biggl(-\alpha u^2+u_x^2-p*(\alpha u^2+u_x^2)\biggr)(t,q(t,x_0))\\
 &=\frac{3\,m(t)^2}{2(\alpha+1)} +R(t),
 \end{split}
\]
where $R(t)=\frac{3}{2(\alpha+1)}(-\alpha u^2-p*(\alpha u^2+u_x^2))(t,q(t,x_0))$. 
Using again $\|u(t)\|_\infty\le c\|u_0\|_{H^1}$ and observing that
$\|p*(\alpha u^2+u_x^2)\|_\infty \le c(\alpha)\|u_0\|_{H^1}^2$ by Young inequality, we get that 
$ |R(t)|$ is uniformly bounded on $(0,T^*)$ by a constant depending only on~$\alpha$ and $\|u_0\|_{H^1}$.

Let $\epsilon>0$. Taking  $0<t_0< T^*$ such that $t_0$ is close enough to $T^*$ in a such way that 
$-\epsilon\le R(t)/m(t)^2\le\epsilon$ on $(t_0,T^*)$, we deduce that, on such interval,
\[
 \frac{3}{2(\alpha+1)}-\epsilon\le \frac{\dd }{\dd t}\biggl(\frac{-1}{m(t)}\biggr) \le \frac{3}{2(\alpha+1)}+\epsilon.
\]
Now integrating these inequalities on $(t,T^*)$ we get the blowup rate~\eqref{burate} with $x(t)=q(t,x_0)$.
\end{proof}

\section{The minimization problem in the limit case~$\beta=\frac{e+1}{e-1}$ and in the case $\beta=1$}
\label{sec:fo}

\subsection{The limit case $\beta=\frac{e+1}{e-1}$}\mbox{}\\
\label{sec:limit}
We will start considering the limit case 
\[\beta=\frac{e+1}{e-1}.\]
In this case, according to formula~\eqref{32}, the weight function becomes
\begin{equation}
\label{agom}
\omega(x)=\frac{2e}{(e-1)^2}\sinh(x), \qquad x\in(0,1).
\end{equation}
Because of Proposition~\ref{prop2}, we are led to assume also
\begin{equation}
\label{asas}
\alpha>-1/C(\textstyle\frac{e+1}{e-1}).
\end{equation}
We start observing that 
\[
 I(\alpha,\textstyle\frac{e+1}{e-1})\ge\widetilde I(\alpha,\textstyle\frac{e+1}{e-1}),
 \]
where
\[
\begin{split}
\widetilde I(\alpha,\textstyle\frac{e+1}{e-1})
&\equiv
\inf\biggl\{ \int_0^1 \bigl( p+\beta p')\Bigl(\alpha u^2+ u_x^2\Bigr)\dd x \colon u\in E_{(e+1)/(e-1)}\;\text{and}\; u(1)=1
\biggr\}\\
&=\alpha+\inf\Bigl\{ J(v) \colon v\in E_{(e+1)/(e-1),0}
  \Bigr\}
\end{split}  
\]
and the functional $J$ is given by~\eqref{JJ}.
Indeed, the inequality $I(\alpha,\frac{e+1}{e-1})\ge\widetilde I(\alpha,\frac{e+1}{e-1})$ follows from the inclusion
$H^1_0(0,1)\subset E_{(e+1)/(e-1),0}$.

The unique minimizer~$v\in E_{(e+1)/(e-1),0}$ of the functional $J$ (whose existence was obtained in the proof of Proposition~\ref{prop1}) satisfies the Euler--Lagrange equation complemented with the right-boundary condition
\begin{equation}
\label{EL}
\begin{cases}
(\omega v_x)_x-\alpha\omega v=\alpha\omega & \text{in $x\in(0,1)$},\\
v(1)=0.
\end{cases}
\end{equation}
Problem~\eqref{EL} is undetermined, but we will see that among its solutions only one belongs to 
$E_{(e+1)/(e-1),0}$.

\subsection{The case $\beta=\frac{e+1}{e-1}$ and $\alpha>-1/C(\frac{e+1}{e-1})$}\mbox{}\\
\label{ssec:leg}
The Euler--Lagrange equation~\eqref{EL} reads
\begin{equation}
\label{ELa}
\sinh(x) v_{xx}(x) + \cosh(x) v_x(x) -\alpha\sinh v(x) = \alpha\sinh(x), \qquad \text{for $x\in(0,1)$.}
\end{equation}
To find the general solution, consider the change of unknown $v(x)=f(y(x))$, with $y=\cosh(x)$.
Then equation~\eqref{ELa} can be rewritten in the $y$ variable as
\begin{equation}
\label{ELf}
 (1-y^2)f_{yy}-2 y f_y(y)+\alpha f(y)=-\alpha, \qquad\text{with $y\in (1,\cosh1)$}.
\end{equation}
The constant function $y\mapsto -1$ is a 
particular solution. 
Substituting $\alpha=\nu(\nu+1)$, we recognize the usual form of a non-homogeneous second order Legendre ODE.
We thus set 
\begin{equation}
\label{def:nu}
\nu(\alpha)=-\frac{1}{2}+\frac{1}{2}\sqrt{1+4\alpha}\in\{z\in\C\colon \mathfrak{Im}(z)\ge0\},
\end{equation}
where the complex square root is taken in $\{z\in\C\colon \mathfrak{Im}(z)\ge0\}$.
The general solution of equation~\eqref{ELf} is thus
\begin{equation}
 \label{genso}
 f_{\lambda,\mu}=-1+ \lambda P_{\nu(\alpha)}(y)
 +\mu Q_{\nu(\alpha)}(y),
 \qquad y\in(1,\cosh 1),
\end{equation}
where $P_\nu$ and $Q_\nu$ are the two associate Legendre functions respectively of the first and of the second kind, of degree~$\nu$.
We recall that  $Q_\nu$ has a logarithmic singularity at $1^+$, whereas $P_\nu$ is bounded as $y\to1^+$.
Moreover, $P_\nu$ is a polynomial when $\nu$ is an integer.
Notice that the function $P_{\nu(\alpha)}(\cosh)$ does belong to $E_{(e+1)/(e-1)}$, but $Q_{\nu(\alpha)}(\cosh )$ does not,
because $Q_{\nu(\alpha)}(\cosh x)\not=O( \sqrt{|\log(x)|})$ as $x\to0^+$.

Hence, the only solution~$\bar v=v_{\lambda,\mu}$
 of~\eqref{ELa}, such that $\bar v\in E_{(e+1)/(e-1),0}$ is obtained taking $\mu=0$ and 
 $\lambda=1/ P_{\nu(\alpha)}(1)$.
 Thus,
 \begin{equation}
 \label{minia}
 \bar v(x)=-1 +\frac{P_{\nu(\alpha)}(\cosh x)}{P_{\nu(\alpha)}(\cosh 1)}, \qquad x\in(0,1).
 \end{equation}
 Expression~\eqref{minia} makes sense provided the denominator is nonzero.
 Therefore, we can recast condition~\eqref{asas} on~$\alpha$, guaranteeing that $I(\alpha,\frac{e+1}{e-1})>-\infty$,
 in the following equivalent form:
 \[
 \alpha>\alpha_0\simeq -6.113,
 \]
 where $\alpha_0$ is the largest zero of the function 
 $\alpha\mapsto{P_{\nu(\alpha)}(\cosh1)}$ (see Figure~\ref{alpha0}).
 The  representation of $P_\nu(z)$ as a hyper-geometric series, convergent for  $|\frac{1}{2}-\frac{z}{2}|<1$, is
 \[
 \begin{split}
 P_\nu(z)
 &={}_2F_{1}\Bigl(-\nu,\nu+1,\frac{1}{2}-\frac{z}{2}\Big)\\
 &=\displaystyle\sum_{k=0}^{+\infty}\frac{\Gamma(-\nu+k)\Gamma(\nu+1+k)}{(k!)^2\Gamma(-\nu)\Gamma(\nu+1)}
 \biggl(\frac{1}{2}-\frac{z}{2}\biggr)^{k}.
 \end{split}
 \]
Such representation holds true for $z=\cosh 1$, but computing $\alpha_0$ analytically seems to be difficult.
On the other hand,  $\alpha_0$ can be estimated via Newton's method.

 For $\alpha>\alpha_0=-1/C(\frac{e+1}{e-1})$, using that $\bar v$ solves~\eqref{EL}, $\bar v_x(0^+)=0$,
 and $\int_0^1\omega=1$, we have
 \begin{equation}
\label{qq1}
\begin{split}
 I(\alpha,\textstyle\frac{e+1}{e-1})
 &\ge\alpha+J(\bar v)
=\alpha+\int_0^1 \omega(\alpha \bar v^2+\bar v_x^2+2\alpha\bar v)\\
&=\alpha+\int_0^1 \bigl[(\omega \bar v_x)_x \bar v +\omega \bar v_x^2 +\alpha\omega \bar v\bigr]\\
&=\alpha+\bigl[\omega \bar v_x \bar v\bigr]_{0^+}^{1^-} + \alpha\int_0^1\omega \bar v\\
&=\int_0^1(\omega \bar v_x)_x= (\omega \bar v_x)(1^-)\\
&=\frac{(e+1)^2}{2e}\,\frac{P_{\nu(\alpha)}'}{P_{\nu(\alpha)}}(\cosh 1).
\end{split}
\end{equation}

\begin{figure}
\begin{minipage}[c]{.46\linewidth}
\includegraphics[width=10cm,height=11cm]{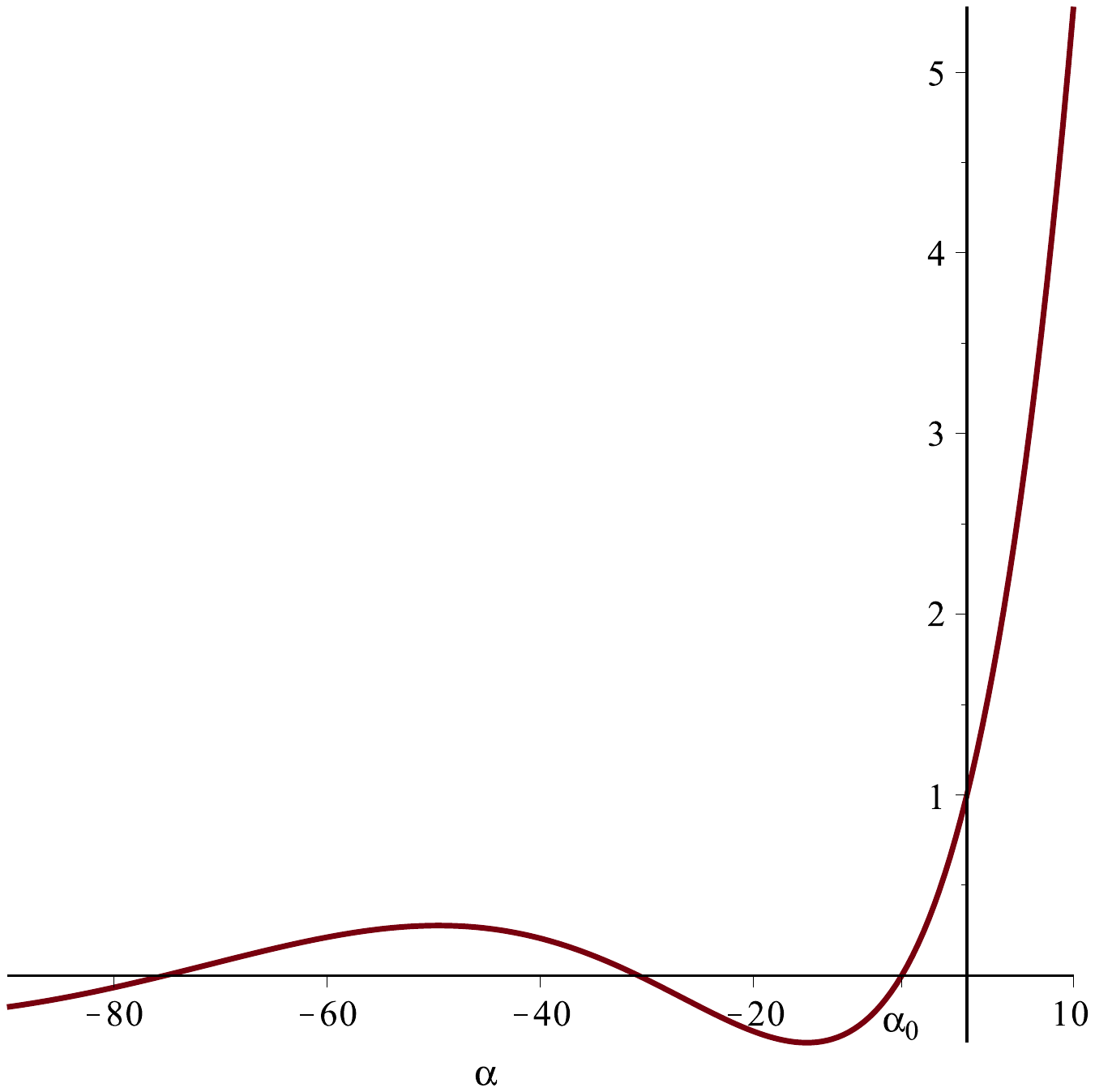}{\centering}
\vskip-5cm
\caption{The function $\alpha\mapsto P_{\nu(\alpha)}(\cosh1)$ and its largest zero~$\alpha_0$.}
\end{minipage}
\hfill
\label{minimizers-limi}
\begin{minipage}[c]{.46\linewidth}
\includegraphics[width=10.5cm,height=11cm]{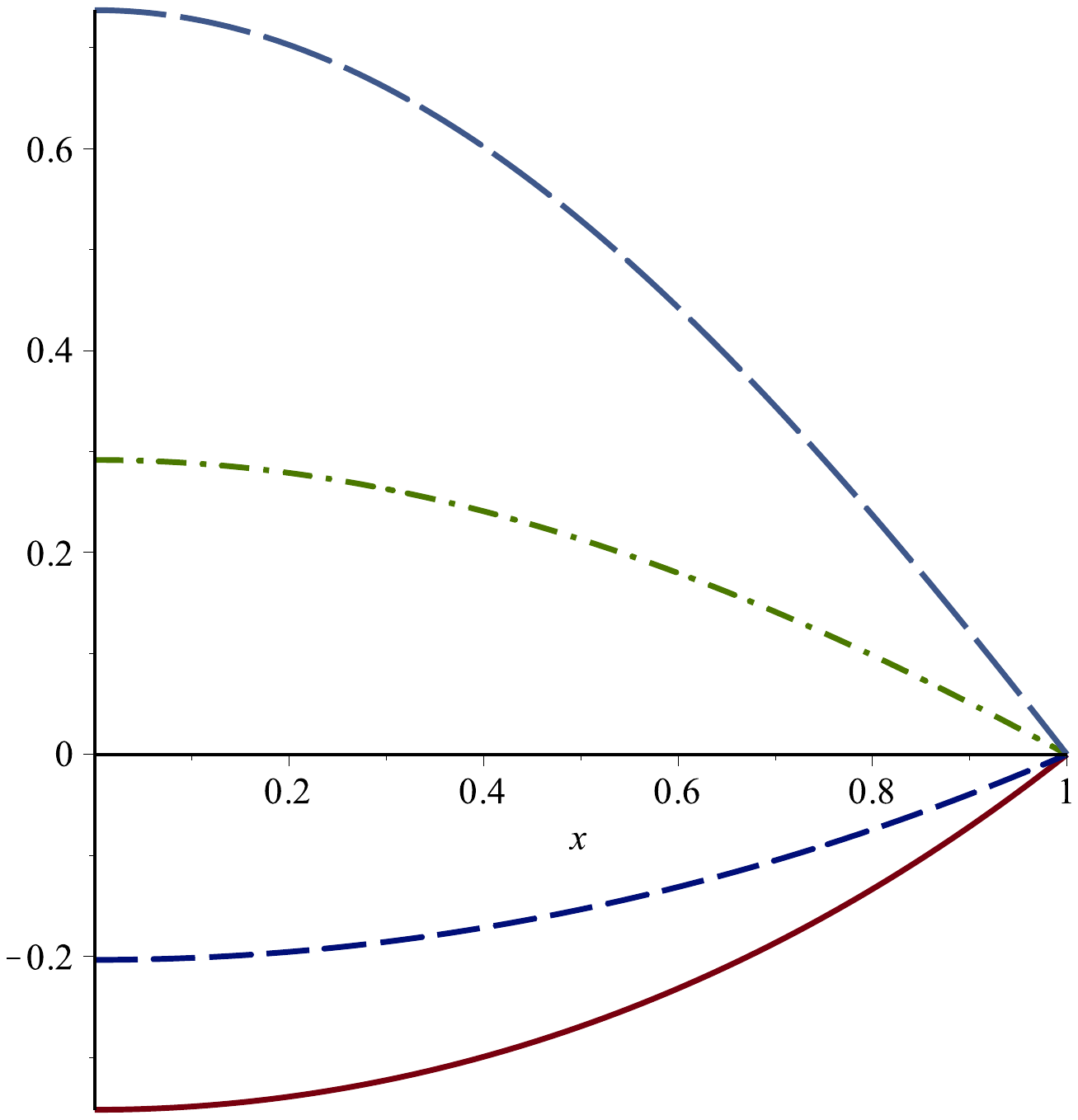}{\centering}
\vskip-4.5cm
\caption{A few minimizers as given by formula~\eqref{qq1}, for $\alpha=2$, $1$, $-1$, $-2$  (from bottom to top).}
\end{minipage}
\label{alpha0}
\end{figure}

\begin{remark}
\label{rem:meco}
Incidentally, we proved that $c=-(1/\alpha_0)\simeq0.164$ is the best constant in the weighted Poincar\'e inequality below (valid for all
$v\in C^1([0,1])$ such that $v(1)=0$,
\[
 \int_0^1 \sinh(x) v(x)^2\dd x\le c\int_0^1 \sinh(x)v_x(x)^2\dd x.
\]
\end{remark}

\subsection{The case $\beta=1$}\mbox{}\\
The computation of $I(\alpha,\beta)$ in the case $0\le \beta<\frac{e+1}{e-1}$ is different than that of the previous section, as standard variational methods apply in the usual Sobolev space $H^1_0(0,1)$.
Moreover, in the case $\beta=1$, the associated Euler--Lagrange boundary value problem
can be explicitly solved.
According to Proposition~\ref{prop1}, the computation below will be valid for $\alpha>-1/C(1)$.
As a byproduct of our calculations, we will find the explicit expression
\[
C(1)=4/(1+4\pi^2) 
\]
for the weighted Poincar\'e inequality.

Recall that
$
I(\alpha,\beta)=\alpha+\inf\{ J(v)\colon v\in H^1_0(0,1)\},
$
where $J(v)$  was introduced in~\eqref{JJ} and the weight function
\begin{equation}
\label{omexp}
  \omega(x)=p(x)+\beta p'(x)=\frac{(1+\beta)e^x+(1-\beta)e^{1-x}}{2(e-1)}, \qquad x\in(0,1)
\end{equation}
is positive and bounded away from~$0$ on the interval $(0,1)$.
The unique minimizer $\bar v\in H^1_0$ is the solution of
\begin{equation}
\label{ELa1}
 \begin{cases}
\omega v_{xx}+\omega_x v_x-\alpha\omega v=\alpha\omega &x\in(0,1)\\
v(0)=v(1)=0.
 \end{cases}
\end{equation}
For $\beta=1$, the weight $\omega$ reduces to $\omega(x)=e^x/(e-1)$.
Thus, the general solution of equation~\eqref{ELa1} for such choice of~$\omega$ (at least for $\alpha\not=-1/4$) is
\begin{equation}
 \label{geso1}
 v(x)=-1+\lambda e^{\nu(\alpha)x}+\mu e^{\overline{\nu(\alpha)}x}, \qquad \lambda,\mu\in\R,
\end{equation}
where the complex number $\nu(\alpha)$ is understood as in~\eqref{def:nu}.
To simplify further the result let us set
\begin{equation}
 \label{def:mu}
 \mu=\frac12\sqrt{1+4\alpha}\in \{z\in\C\colon\mathfrak{Im}(z)\ge0\}.
\end{equation}
Imposing the boundary conditions~$v(0)=v(1)=0$ we finally get the expression of minimizer of $I(\alpha,1)$:
\[
 \bar v(x)=-1+\frac{\sqrt e\sinh(\mu(\alpha)x)+\sinh(\mu(\alpha)(1-x))}{e^{x/2}\sinh(\mu(\alpha))}.
\]
The minimum $I(\alpha,1)$ is thus given by
\begin{equation}
\label{ia1}
\begin{split}
I(\alpha,1)
&=\alpha+J(\bar v)= \omega(1)\bar v_x(1)-\omega(0)\bar v_x(0)\\
&=-\frac12 + \mu\cdot\frac{\cosh\frac12\cosh\mu-1 }{\sinh\frac12\sinh\mu}.
\end{split}
\end{equation}
The above expression makes sense provided $\sinh(\mu(\alpha))\not=0$,
{\it i.e.\/} for $\alpha\not=-\frac14-k^2\pi^2$, with $k\in\Z$.
But $\alpha=-1/4$ is in fact a removable singularity in~\eqref{ia1}. The restriction to be imposed on~$\alpha$ is thus
$\alpha>-\frac14-\pi^2$.

\section{Proof of Theorem~\ref{theoquant}}
\label{sec:tq}

We are now in the position of proving Theorem~\ref{theoquant}.

\begin{proof}[Proof of Theorem~\ref{theoquant}]
Let us recall the definition of $\beta_\gamma$,
\begin{equation}
\label{def2:beta}
\beta_\gamma=\inf\Bigl\{ \beta\in\R^+\colon 
\beta^2+ I(\alpha,\beta)-\alpha\ge0\Bigr\},
\end{equation}
where the one-to-one relation between $\alpha$ and $\gamma$ is
\begin{equation}
\label{gal}
 \alpha=\frac{3-\gamma}{\gamma},\qquad\text{or}\qquad\gamma=\frac{3}{1+\alpha}.
\end{equation}
Using the results of the previous section we can now give explicit bounds from below for $I(\alpha,\beta)$ that can be used for the estimate of $\beta_\gamma$.

Using  the concavity properties of the function~$\beta\mapsto I(\alpha,\beta)$ we see that, for any fixed
 $\alpha>\alpha_0$, the infimum
$I(\alpha,\beta)$ is bounded from below by piecewise affine function
of the~$\beta$ variable. Namely,
\begin{equation}
\label{affi}
\begin{cases}
I(\alpha,\beta)\ge I(\alpha,1), &\text{if $0\le\beta\le1$}\\
I(\alpha,\beta)\ge \frac{e-1}{2}\Bigl(I(\alpha,\textstyle\frac{e+1}{e-1})-I(\alpha,1)\Bigr)\beta
   +\frac{e+1}{2}I(\alpha,1)-\frac{e-1}{2}I(\alpha,\frac{e+1}{e-1}),
&\text{if $1\le \beta\le \frac{e+1}{e-1}$}.
 \end{cases}
\end{equation}
We denote by $R(\alpha,\beta)$, the function defined by the right-hand in~\eqref{affi}, for $\alpha>\alpha_0$ and
$0\le\beta\le \frac{e+1}{e-1}$.
The condition $\alpha>\alpha_0$ ensures that $I(\alpha,1)\ge I(\alpha,\textstyle\frac{e+1}{e-1})>-\infty$, so that 
$R(\alpha,\beta)$ is finite.

The first issue is to find the condition on $\alpha$ guaranteeing $\beta_\gamma<+\infty$.
Owing to the lower bound~\eqref{affi}, a sufficient condition for this is that $\alpha$ is chosen in a such way that
\begin{equation}
\label{sc1}
\begin{split}
\textstyle
 \exists \, \beta\;\text{such that}\; 0\le\beta\le\frac{e+1}{e-1}\colon \quad\beta^2+ R(\alpha,\beta)-\alpha\ge0
\end{split}
\end{equation}
But condition~\eqref{sc1} is equivalent to the following one:
\begin{equation}
  1+I(\alpha,1)-\alpha\ge0\quad\text{or}\quad  
         \Bigl(\frac{e+1}{e-1}\Bigr)^2+I\bigl(\alpha,\textstyle\frac{e+1}{e-1}\bigr)-\alpha \ge0.
\end{equation}
Indeed, one implication is obvious and the converse one easily follows applying to the $\beta$-variable the elementary properties of quadratic polynomials.
Let us make more explicit the last condition: because of the definition of~$I(\alpha,\beta)$ and 
Proposition~\eqref{prop2}, the two functions 
$\alpha\mapsto I(\alpha,\frac{e+1}{e-1})$ and $\alpha\mapsto I(\alpha,1)$
are both increasing, concave and vanishing at $\alpha=0$.
Therefore,  there exist $\alpha_1^-<0<\alpha_1^+$ such that
\begin{equation}
 \label{def:alpha1}
 \Bigl(\frac{e+1}{e-1}\Bigr)^2+I\bigl(\alpha,\textstyle\frac{e+1}{e-1}\bigr)-\alpha \ge0 \;\iff\;
 \alpha_1^-\le\alpha\le\alpha_1^+.
\end{equation}
For the same reason,  there exist $\alpha_2^-<0<\alpha_2^+$ such that
\begin{equation}
 \label{def:alpha2}
 1+I(\alpha,1)-\alpha \ge0 \;\iff\;
 \alpha_2^-\le\alpha\le\alpha_2^+.
\end{equation}
The above zeros can be easily estimated via Newton's method. We find in this way
\[
\alpha_1^-<\alpha_2^-<0<\alpha_2^+<\alpha_1^+,
\]
see also Figures~\ref{minimum-limit}-\ref{minimum-limit1}.

\begin{figure}
\begin{minipage}[c]{.49\linewidth}
\includegraphics[width=10cm,height=11cm]{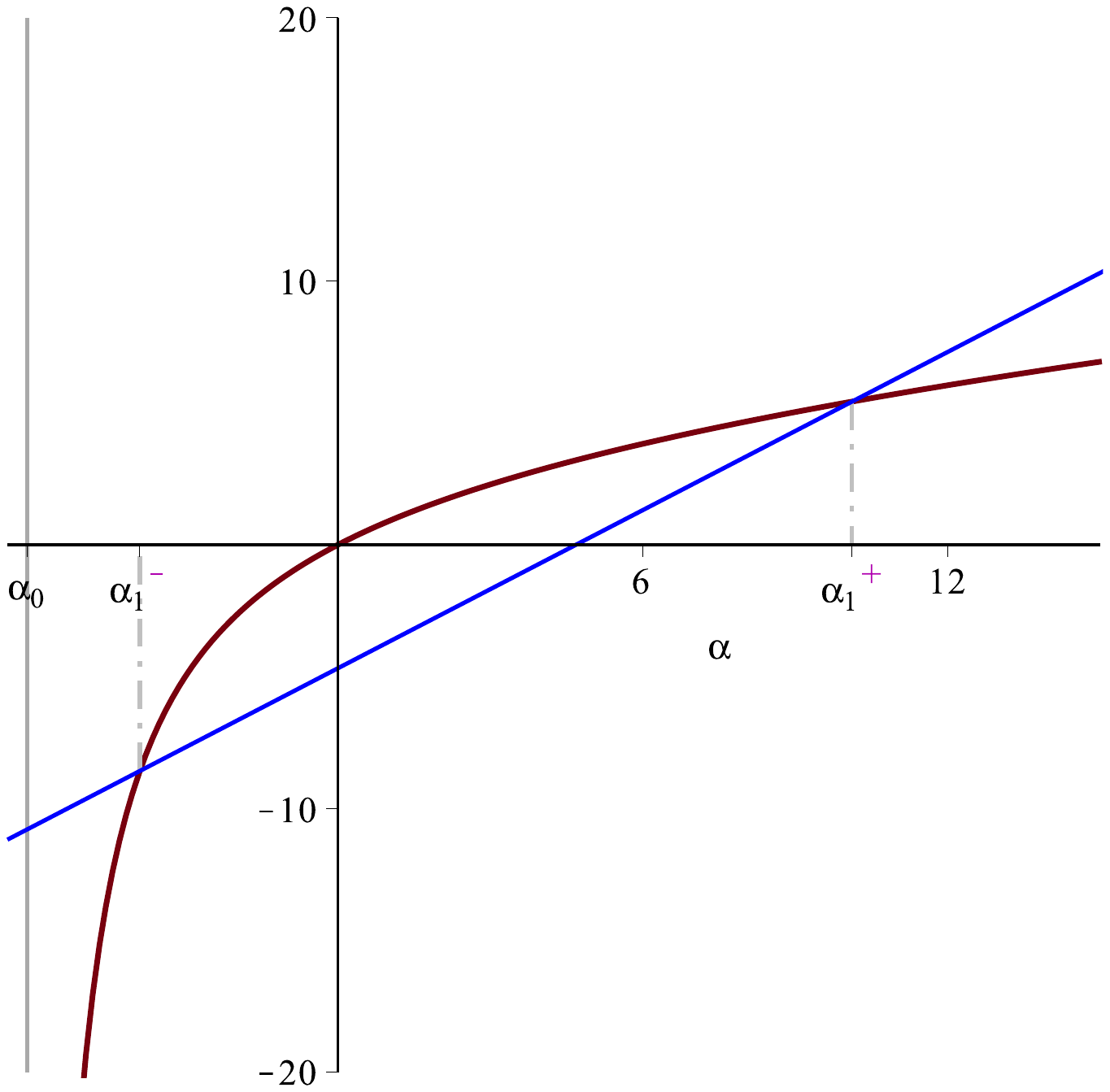}{\centering}
\vskip-4.5cm
\caption{The plot of $\alpha\mapsto I(\alpha,\frac{e+1}{e-1})$ (see equation~\eqref{qq1}) and of the straight line $\alpha\mapsto \alpha-\bigl(\frac{e+1}{e-1}\bigr)^2$, intersecting the curve at $\alpha_1^-$ and $\alpha_1^+$.}
\end{minipage}
\hfill
\label{minimum-limit1}
\begin{minipage}[c]{.49\linewidth}
\includegraphics[width=10.5cm,height=11cm]{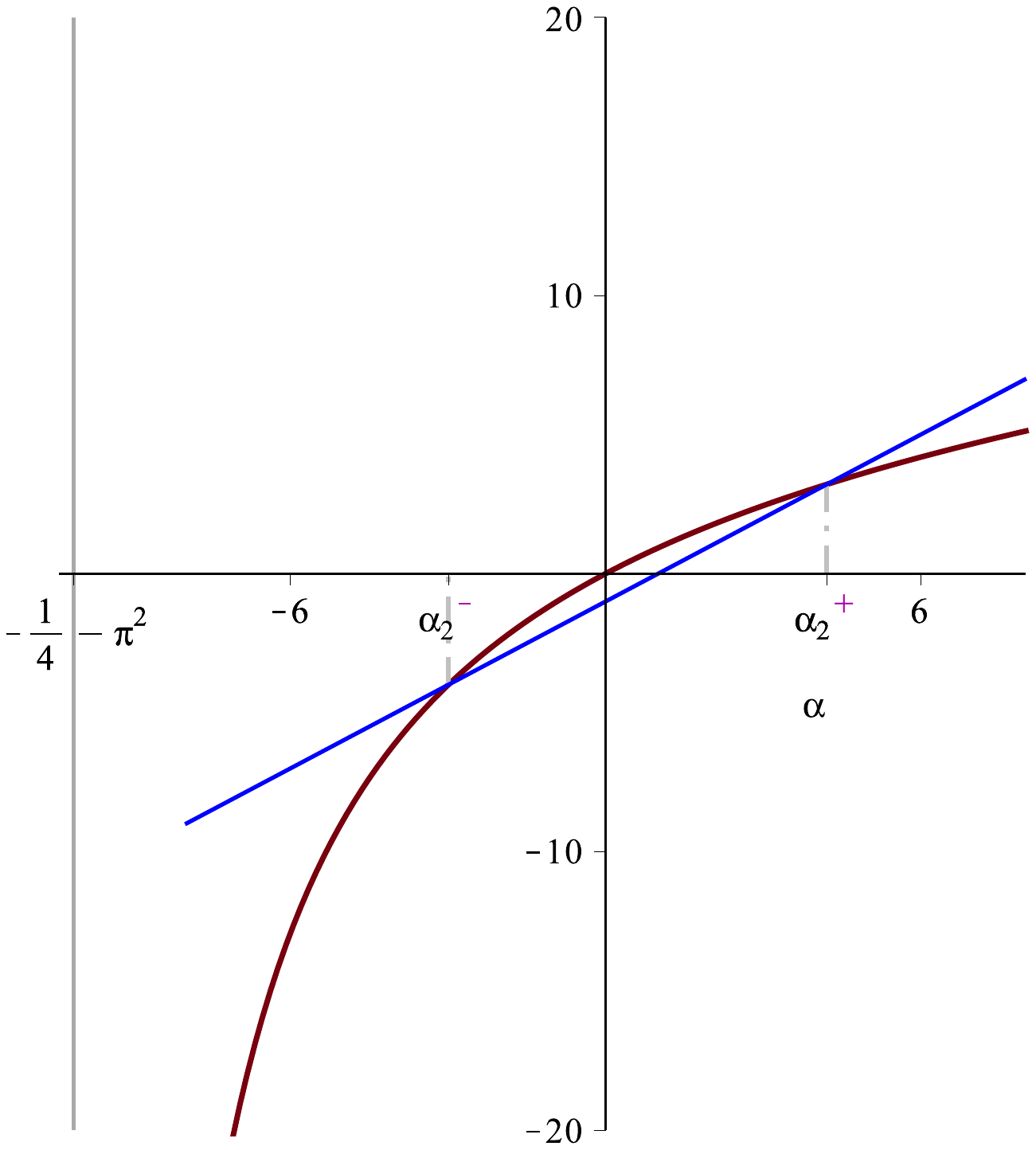}{\centering}
\vskip-4.5cm
\caption{Plot of $\alpha\mapsto I(\alpha,1)$ (see equation~\eqref{ia1}) and of the straight line $\alpha\mapsto \alpha-1$, intersecting the curve at $\alpha_2^-$ and $\alpha_2^+$.}
\end{minipage}
\label{minimum-limit}
\end{figure}

According to~\eqref{gal}, let us introduce the four constants
\begin{equation}
 \label{4ga}
\begin{matrix}
 \gamma_1^-=\frac{3}{1+\alpha_1^-}=-1.036\ldots \quad
 &\gamma_1^+=\frac{3}{1+\alpha_1^+}=0.269\ldots\\
 \gamma_2^-=\frac{3}{1+\alpha_2^-}=-1.508\dots\quad
& \gamma_2^+=\frac{3}{1+\alpha_2^+}=0.575\dots
\end{matrix}
\end{equation}

In particular, for $\alpha_1^-\le\alpha\le\alpha_1^+$, we find that the set in~\eqref{def2:beta} is nonempty: this means that in the range $\gamma\in (-\infty,\gamma_1^-]\cup[\gamma_1^+,+\infty)$ we get
 $\beta_\gamma<+\infty$, and Theorem~\ref{maintheo} and Corollary~\ref{cnge} apply.

On the other hand, under the more restrictive condition $\alpha_2^-\le\alpha\le\alpha_2^+$, 
we can make use
of inequality~\eqref{def:alpha2} to get the estimate
\begin{equation}
 \beta_\gamma\le \sqrt {\alpha-I(\alpha,1)} \le1.
\end{equation}
Now recalling the expression of~$I(\alpha,1)$ computed in~\eqref{ia1} we obtain the explicit estimate,
valid for
$\gamma\in (-\infty,\gamma_2^-]\cup[\gamma_2^+,+\infty)$:
\begin{equation}
 \label{bob1}
 \beta_\gamma\le\sqrt{\frac{3}{\gamma}-\frac12
  - \mu\cdot\frac{\cosh\frac12\cosh\mu-1 }{\sinh\frac12\sinh\mu}},
\end{equation}
where $\mu$ is the complex number
\[
 \mu=\frac12\sqrt{1+4(3-\gamma)/\gamma}.
\]
The choice between the two complex square roots of $1+4(3-\gamma)/\gamma$ does not affect the result
and the radical in~\eqref{bob1} (or in~\eqref{est:betb} below) is nonnegative because of the last inequality in~\eqref{maxi}.
Observe that taking here $\gamma=1$ gives the blowup criterion~\eqref{buch} for the Camassa--Holm equation.

If otherwise $\alpha\in[\alpha_1^-,\alpha_1^+]\backslash[\alpha_2^-,\alpha_2^+] $ then,
from the inequality~\eqref{affi} $I(\alpha,\beta)\ge R(\alpha,\beta)$, we obtain the bound 
$\beta_\gamma\le\tilde\beta_\gamma$, where $\tilde \beta_\gamma$ is the only  
zero of the quadratic polynomial $\beta\mapsto\beta^2+R(\alpha,\beta)-\alpha$ inside the interval 
$[1,\frac{e+1}{e-1}]$.
Thus, letting
\begin{equation}
\label{defin:b} 
\begin{split}
b&\equiv \frac{e-1}{2}\Bigl(I(\alpha,\textstyle\frac{e+1}{e-1})-I(\alpha,1)\Bigr)\\
  &=-\frac{e-1}{2}
 \biggl(
   \frac{(e+1)^2}{2e}\,\frac{P_{-1/2 +\mu}'}{P_{-1/2+\mu}}(\cosh 1) 
   +\frac12 - \mu\cdot\frac{\cosh\frac12\cosh\mu-1 }{\sinh\frac12\sinh\mu}
  \biggr).
\end{split}
\end{equation}
and
\begin{equation}
\label{defin:c} 
\begin{split}
c&\equiv \textstyle\frac{e+1}{2}I(\alpha,1)-\frac{e-1}{2}I(\alpha,\frac{e+1}{e-1})-\alpha\\
 &=\frac{e+1}{2}\biggl( -\frac12 + \mu\cdot\frac{\cosh\frac12\cosh\mu-1 }{\sinh\frac12\sinh\mu}\biggr)
  -\frac{e-1}{2}
  \biggl(
   \frac{(e+1)^2}{2e}\cdot\frac{P_{-1/2 +\mu}'}{P_{-1/2+\mu}}(\cosh 1)
  \biggr),
\end{split}
\end{equation}
we get the estimate, valid for $\gamma\in[\gamma_2^-,\gamma_1^-]\cup\gamma\in[\gamma_1^+,\gamma_2^+]$:
\begin{equation}
 \label{est:betb}
 \beta_\gamma\le -\frac{b}{2}+\frac12\sqrt{b^2-4c}.
\end{equation}

By our construction, for $\gamma=\gamma_2^-$ or $\gamma=\gamma_2^+$, the equality holds in~\eqref{bob1} and
in~\eqref{est:betb} and we have in this case $\beta_{\gamma_2^-}=\beta_{\gamma_2^+}=1$.

The limit $\beta_\infty=\lim_{\gamma\to\infty}\beta_\gamma$ is obtained taking $\alpha=-1$ in~\eqref{def2:beta}.
Going back to the estimate~\eqref{bob1}, letting $\gamma\to\infty$ we get
\begin{equation}
 \label{asbb}
 \beta_\infty\le 
 \sqrt{
 \frac{\sqrt3\bigl(1-\cosh\frac12\cos\frac{\sqrt3}{2}\bigr)}{2\sinh\frac12\sin\frac{\sqrt3}{2}}  -\frac12   }.
\end{equation}

The three claims of Theorem~\ref{theoquant} are now established.
\end{proof}

\section{The case of the Camassa--Holm equation}
\label{sec:ch}

In the case of the Camassa--Holm equation ($\alpha=2$),
 it is remarkable that the Euler--Lagrange equation 
\begin{equation}
 \label{el2}
 (\omega_xv_x)_x-2\omega v=2\omega, \qquad x\in(0,1),
\end{equation}
associated with the minimization
of $I(2,\beta)$ can be explicitly solved {\it for any\/} $\beta$.
 
Indeed, observing that from~\eqref{omexp} we have $\omega=\omega_{xx}$, on $(0,1)$, 
the associate homogeneous equation $(\omega v_{x})_x-2\omega v=0$ clearly possess
$v(x)=\omega_x(x)$ as a solution. We can compute a second independent  solution of this homogeneous equation of the form $v(x)=\omega_x f(x)$. Then $f$ must satisfy
$f_{xx}+(\frac{2\omega}{\omega_x}+\frac{\omega_x}{\omega})f_x=0$, provided $\omega\omega_x\not=0$.
Thus, $f_x(x)=\frac{1}{\omega\omega_x^2}$ in intervals where $\omega_x\not=0$.
Observe that
\[
\begin{split}
&\text{if $\beta>\frac{e-1}{e+1}$, then $\omega_x>0$ on the interval $(0,1)$}\\
&\text{if $0<\beta<\frac{e-1}{e+1}$, then $\omega_x\not=0$ for 
$x\in(0,1)$ and $x\not=x_\beta\equiv\frac{1}{2}\log(\textstyle\frac{(1-\beta)e}{1+\beta})$.}
\end{split}
\]
On the other hand, using again $\omega=\omega_{xx}$ and integration by parts in the indefinite integrals below we see that,
for $x\not=x_\beta$ and an arbitrary constant $C\in\R$
\begin{equation*}
\begin{split}
 \omega_x(x)\int\frac{1}{\omega\omega_x^2}
 &=-\omega_x(x)\int\frac{1}{\omega^2}\cdot\biggl(\frac{1}{\omega_x}\biggr)_x\\
 &=-\frac{1}{\omega(x)^2}+C\omega_x(x)-2\omega_x(x)\int\frac{1}{\omega^3}
 \end{split}
\end{equation*}
The expression on the right-hand side is well defined on the whole interval $(0,1)$, for all $0\le\beta<\frac{e-1}{e+1}$.
Therefore, the general solution of~\eqref{el2} is
\begin{equation}
 v(x)=-1+\omega_x(x)\biggl(\lambda+2\mu\int\frac{1}{\omega^3}\biggr)+\frac{\mu}{\omega(x)^2},
 \qquad x\in(0,1).
\end{equation}

We compute $\int\frac{1}{\omega^3}$ making the change of variables $y=e^x$:
\begin{equation*}
\begin{split}
\int\frac{1}{\omega(x)^3}\dd x=\frac{(e-1)^3}{(1+\beta)^3} \biggl[\int\frac{8y^2}{(y^2+B)^3}\dd y\biggr]_{y=e^x}, 
\end{split}
\end{equation*}
with 
\[
B=\frac{e(1-\beta)}{1+\beta}.
\]
The last  integral can be easily computed distinguishing  the cases $B>0$, $B=0$ and $B<0$.
For example, in the case $B>0$, that corresponds to $0<\beta<1$, we have
 \begin{equation*}
\begin{split}
\int\frac{1}{\omega(x)^3}\dd x
&= \frac{(e-1)^3}{(1+\beta)^3}\biggl(\frac{y(y^2-B)}{B(y^2+B)^2}+B^{-3/2}\arctan(t/\sqrt B)\biggr)\\
&=\frac{2\sinh(1/2)^2}{1-\beta^2}\cdot\frac{\omega_x(x)}{\omega(x)^2}
+\frac{8\sinh(1/2)^3}{(1-\beta^2)^{3/2}}\arctan\Bigl(\sqrt{\textstyle\frac{1+\beta}{e(1-\beta)}}\,e^x\Bigr).
\end{split}
\end{equation*}

The minimizer $\bar v$ of $I(\alpha,\beta)$, with $0\le\beta<\frac{e+1}{e-1}$ 
is obtained choosing in the above expression the coefficients $\bar\lambda$ and $\bar\mu$ solving the linear system
\begin{equation*}
\begin{cases}
 \omega_x(0)\bar\lambda+\frac{1}{\omega(0)}\bar\mu=1\\
 \omega_x(1)\bar\lambda + \biggl(2\omega_x(1)\Bigl(\frac{1}{\omega(1)^2}+\int_0^1\frac{1}{\omega^3}\Bigr)\biggr)\bar\mu=1.
\end{cases}
\end{equation*}
Once $\bar v$ is computed as indicated, the minimum is given, for $0\le\beta<\textstyle\frac{e+1}{e-1}$, by
\begin{equation}
\begin{split}
I(2,\beta)
&=2+J(\bar v)= \omega(1)\bar v_x(1)-\omega(0)\bar v_x(0)\\
&=\omega(1)^2\biggl(\bar\lambda+2\bar\mu\int_0^1\frac{1}{\omega^3}\biggr)-\bar\lambda\omega(0).
\end{split}
\end{equation}
The explicit expression  of $I(2,\beta)$ is thus easily written in  terms of elementary functions. We do not reproduce the complete formula here as it is too long to be really useful.
We just provide here a few particular values and its plot in the interval $[0,\frac{e+1}{e-1}]$
 (see Figure~\ref{i2beta}).
\begin{equation}
\label{values}
\begin{split}
I(2,0)&=1+\frac{\arctan(\sinh\frac12)}{\sinh\frac12+\arctan(\sinh\frac12)(\sinh\frac12)^2}=1.737\ldots\\
I(2,1)&=-\frac12+\frac32\cdot\frac{\cosh\frac12\cosh\frac32-1}{\sinh\frac12\sinh\frac32}= 1.734\ldots\\
I(2,\textstyle\frac{e+1}{e-1})&=\frac{(e+1)^2}{e^2+1}= 1.648\ldots
\end{split}
\end{equation}
The value of $I(2,0)$ was computed also in~\cite{Wah-NoDEA07} or \cite{Zhou-Calc-Var}.

\begin{figure}
\includegraphics[width=14cm,height=13cm]{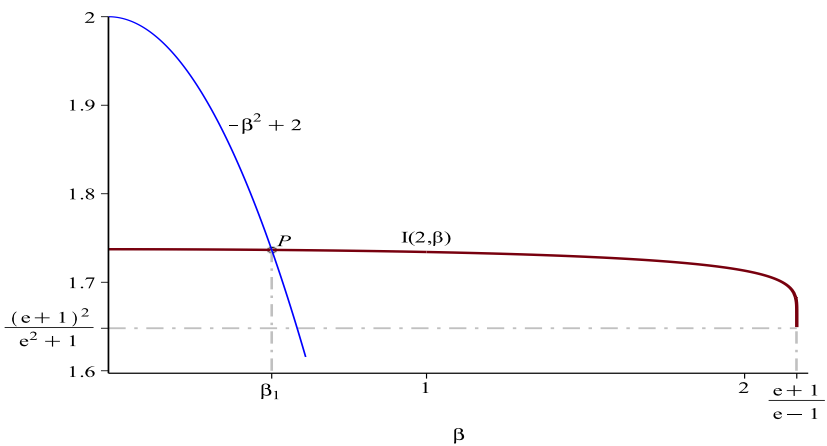}{\centering}
\vskip-6.5cm
\caption{The plot of $I(2,\beta)$ and the construction of $\beta_1=0.513\ldots$}
\label{i2beta}
\end{figure}

\begin{remark}
The last formula in~\eqref{values}, in fact, is obtained in a slight different way, as in our previous computations
we excluded the limit case $\beta=\frac{e+1}{e-1}$.
Of course, this formula agrees with that obtained from the more general one~\eqref{qq1} for $\alpha=2$, 
 as $\nu(2)=1$  and $P_1(y)=y$.
But we can easily prove it without relying on the more complicated approach used to get~\eqref{qq1}.
Indeed, when $\beta=\frac{e+1}{e-1}$, we can write the general solution
of~\eqref{el2} as 
\begin{equation}
v_{\lambda,\mu}(x)=-1+\lambda \cosh(x)+\mu \biggl(1+\frac{1}{2}\cosh(x)\log\biggl(\frac{\cosh(x)-1}{\cosh(x)+1}\biggr)\biggr),
 \qquad x\in(0,1).
\end{equation}
The appropriate boundary condition for $\beta=\frac{e+1}{e-1}$ is $v_{\lambda,\mu}\in E_{(e+1)/(e-1),0}$.
The only possibility to achieve this is to choose $\mu=0$ 
(recall that the elements of $E_{(e+1)/(e-1),0}$ are $O(\sqrt{|\log(x)|})$ as $x\to0+$).
The remaining condition $v(1)=0$ requires $\lambda=1/\cosh(1)$.
The minimizer of the functional $J$ in the case $\alpha=2$ and $\beta=\frac{e+1}{e-1}$, is then
\begin{equation}
\label{miniCH}
\bar v(x)=-1+\textstyle\frac{\cosh(x)}{\cosh(1)}.
\end{equation}
This indeed leads to the last formula in~\eqref{values}.
\end{remark}

In view of the application of Theorem~\ref{maintheo}.
it is interesting to make the numerical study of the zeros of the function
$\beta\mapsto \beta^2+ I(2,\beta)-2$, on the interval $[0,\frac{e+1}{e-1}]$.
The only zero is $\beta_1= 0.513\ldots$
According to Theorem~\ref{maintheo}, solutions of Camassa--Holm blowup 
provided $u_0'(x_0)<-\beta_1|u_0(x_0)|$.
On the other hand, Theorem~\ref{theoquant} predicted $\beta_1\le 0.515\ldots$ (see criterion~\eqref{buch}).
This is another confirmation that the estimate provided by Theorem~\ref{theoquant} is almost sharp.

To go one step further,  one might ask what is the {\it best constant\/} $\beta_1^*$ with the property that
if 
\[\inf_{x\in \S} \Bigl(u_0'(x)+\beta_1^*|u_0(x)|\Bigr)<0,\] 
then the solution of the Camassa--Holm equation arising from~$u_0$ blows up
in finite time.
We do not know how to exactly compute $\beta_1^*$ in the periodic case. However, the following estimates hold:
\[
0.462\dots=\textstyle\frac{e-1}{e+1}\le \beta_1^*\le \beta_1=0.513\ldots
\]
To establish the lower bound, we can make use of a property specific of the case $\gamma=1$. Namely, the fact that if the initial potential $y_0=u_0-u_{0,xx}$ has a constant sign then the corresponding solution exist globally. See \cites{ConstAIF00, Dan01}.
Let us take a sequence $y_0^n$ of smooth, periodic, non-negative functions converging in the distributional sense to a Dirac comb.
Then the  corresponding initial data $u_0^n=p*y_0^n$ give rise to global smooth periodic solutions.
But $\frac{u_{0,x}^n(0)}{u_0^n(0)}\to -\frac{\sinh(1/2)}{\cosh(1/2)}=-\frac{e-1}{e+1}$. Therefore, a condition of the form
$u_0'(x_0)<-c|u_0(x_0)|$, in general, does  not imply the blowup, unless $c\ge\frac{e-1}{e+1}$.

\section{Proof of Corollary~\ref{cnge} and the case of non-periodic solutions.}
\label{sec:cor}

\begin{proof}[Proof of Corollary~\ref{cnge}]
We assume that $\gamma\ge\gamma_1^+$, so in particular that $\gamma>0$.
 In the case $\gamma\le\gamma_1^-$ the proof is similar.
From the assumption of the corollary, $T^*=+\infty$. Then Theorems~\ref{maintheo}-\ref{theoquant} apply, 
implying that
 for all~$t\ge0$, and all~$x\in\R$ we have 
$u_x(t,x)\ge -\beta_\gamma|u(t,x)|$.
If $[a,b]$ is an interval where $u(t,x)\ge0$ for all $x\in[a,b]$, then we have 
\begin{align*}
 -\int_a^b\beta_\gamma e^{\beta_\gamma x}u(t,x)\dd x 
 &\le \int_a^be^{\beta_\gamma x}u_x(t,x)\dd x\\
 &=e^{\beta_\gamma b}u(t,b)-e^{\beta_\gamma a}u(t,a)-\beta_\gamma\int_a^be^{\beta_\gamma x}u(t,x)\dd x.
\end{align*}
This implies that
$e^{\beta_\gamma a}u(t,a)\le e^{\beta_\gamma b}u(t,b)$. 
 
If $u(t,\cdot)\le0$ on $[a,b]$, then applying again Theorem~\ref{maintheo} we have $u_x(t,x)\ge \beta_\gamma u(t,x)$.
This in turn implies
\begin{align*}
 \int_a^b \beta_\gamma e^{-\beta_\gamma x}  u(t,x)\dd x 
 &\le \int_a^b e^{-\beta_\gamma x}u_x(t,x)\dd x\\
 &=e^{-\beta_\gamma b}u(t,b)-e^{-\beta_\gamma a}u(t,a)+\beta_\gamma\int_a^be^{-\beta_\gamma x}u(t,x)\dd x.
\end{align*}
We then conclude that
$e^{-\beta_\gamma a}u(t,a)\le e^{-\beta_\gamma b}u(t,b)$.

Summarizing, we proved that $x\mapsto e^{\beta_\gamma x} u(x,t)$ is monotone increasing in any interval where
$u(\cdot, t)\ge0$ and $x\mapsto e^{-\beta_\gamma x} u(x,t)$ is monotone increasing in any interval where $u(\cdot,t)\le 0$. The condition $u(x_0,t_0)=0$ together with the periodicity of~$u$ imply that $u(\cdot,t_0)\equiv0$ and so $u\equiv0$ because of the conservation of the $H^1$-norm.
\end{proof}

As a consequence of this Corollary, we can deduce that if $u_0$ is not identically zero, but $\int_\S u_0=0$, then the corresponding solution must blow up in finite time.
Indeed, the zero-mean condition implies of course that $u_0$ must vanish in some point $x_0\in\S$.
We recover in this way a known conclusion for the Camassa--Holm equation (see \cite{ConEschCPAM}), extended
for the rod equation (at least for $\gamma$ outside a neighborhood of the origin) in~\cite{Hu-Yin10}.

\subsection*{A digression on non-periodic solutions}
The above proof applies also global solutions 
$u\in C([0,\infty),H^s(\R))\cap C([0,\infty),H^{s-1}(\R))$ of the rod equation~\eqref{rod}.
Notice that in this case the kernel~$p$ of
$(1-\partial_x^2)^{-1}$ is given by
\[
p(x)=\frac12 e^{-|x|}, \qquad x\in\R.
\]
Indeed, let us recall that by the result in~\cite{BraCMP}, if $\gamma\in[1,4]$ and $u_0\in H^s(\R)$, with $s>3/2$,
is such that
\begin{equation}
\label{crir}
 \exists\,x_0\in\R\quad\text{such that}\quad u_0'(x_0)
 < -\sqrt{-\frac12+\frac3\gamma-\frac{\sqrt{12-3\gamma}}{2\sqrt\gamma}\,}\,|u(x_0)|
\end{equation}
then  the unique local-in-time solution $u\in C([0,T],H^s(\R))\cap C^1([0,T],H^{s-1}(\R))$ must blow up
in finite time. In the Camassa--Holm case $\gamma=1$ the above coefficient equals $-1$
and is known to be optimal, see \cite{Dan01}.
Reproducing  the  proof of Corollary~\ref{cnge} in the case of non-periodic solution,
taking
\begin{equation}
\label{begar}
 \beta_\gamma=\sqrt{-\frac12+\frac3\gamma-\frac{\sqrt{12-3\gamma}}{2\sqrt\gamma}\,}
\end{equation}
we get the following result.

\begin{corollary}
 \label{cnge2}
 Let $\gamma\in[1,4]$ and $s>3/2$.
 Let  $u\in C([0,+\infty),H^s(\R))\cap C^1([0,+\infty),H^{s-1}(\R))$, be a global solution 
 of the rod equation~\eqref{rod} (with $p(x)=\frac12 e^{-|x|}$ and $x\in\R$).
 \begin{enumerate}
\item[i)] For all $t\ge0$,
either $u(t,x)>0$ for all $x\in\R$, or $u(t,x)<0$  for all $x\in\R$, or 
$\exists \,x_t\in\R$ such that $u(t,\cdot)\le 0$ in $(-\infty,x_t]$ and $u(t,x)\ge 0$ in $[x_t,+\infty)$.
 In the latter case, if $x\mapsto u(t,x)$ vanishes at two distinct points of the real line, then $x\mapsto u(t,x)$
 must vanish in the whole interval between them.
\item[ii)] If, at some time $t_0\ge0$, 
\[ \liminf_{x\to+\infty}e^{\beta_\gamma x}u(t_0,x)\le0
  \quad\text{and}\quad
  \limsup_{x\to-\infty}e^{-\beta_\gamma x}u_0(t_0,x)\ge0,\]
where $\beta_\gamma$ is given by equation~\eqref{begar},
then $u$ is the identically zero solution.
\end{enumerate}
In particular, if $0\not\equiv u_0\in H^s(\R)$ is such that 
$u_0(x)=o\bigl(e^{-\beta_\gamma|x|}\bigr)$
for $|x|\to\infty$, then the corresponding solution of the rod equation must blow up in finite time.
\end{corollary}
Indeed, all these claims follow from the fact that if $u$ is a global solution than
$x\mapsto e^{\beta_\gamma x} u(x,t)$ is monotone increasing in any interval where
$u(\cdot, t)\ge0$ and $x\mapsto e^{-\beta_\gamma x} u(x,t)$ is monotone increasing in any interval where $u(\cdot,t)\le 0$, as seen before.

\begin{remark}
In the last conclusion of the corollary we recover the known fact 
that solutions of the Camassa--Holm equation
arising from compactly supported data (see \cite{HMPZ07}), or more in general  (see~\cite{Bra11}) from data decaying 
faster than peakons ---solitons with profile $ce^{-|x|}$---
feature a wave breaking phenomenon after some time. The proofs given in~\cites{Bra11,HMPZ07}
relied on McKean's necessary and sufficient condition  \cite{McKean04} for wave breaking and were not suitable for the generalization to $\gamma\not=1$.
Moreover, the sign condition on global solutions~$u$ contained in the first item of our corollary is in the same spirit as the sign condition on the associated potential $u-u_{xx}$ provided by McKean's theorem.
\end{remark}

The method that we used in our previous paper~\cite{BraCMP} in the non-periodic case  looks  simpler than that 
of the present paper. Indeed, in ~\cite{BraCMP} we relied on elementary estimates for bounding from below the convolution 
term $p*(\alpha u^2+u_x^2)$ without making use of variational methods.
On the other hand, the result obtained therein is much weaker as the range of applicability for the 
parameter~$\gamma$ is considerably narrower.
We point out that applying the variational method of the present paper in the non-periodic case we can recover, but not improve,
the results in~\cite{BraCMP}.
Indeed, the main issue is the study of the minimization problem (the analogue of~\eqref{iab} but now
with $p(x)=\frac12e^{-|x|}$):
\begin{equation}
\label{iabr}
I_\R(\alpha,\beta)=
\inf\biggl\{ \int_\R \bigl(p+\beta p_x)\Bigl(\alpha u^2+ u_x^2\Bigr)\dd x \colon u\in H^1(\R),
\; u(0)=1\biggr\}.
\end{equation}
The analogue of Proposition~\eqref{prop1} and Proposition~\ref{prop2} in our present setting
is provided by the following one:
\begin{proposition}
We have
 \label{prop33}
 \begin{equation}
I_\R(\alpha,\beta)>-\infty \iff \alpha\ge-\frac{1}{4} \quad\text{and}\quad -1\le\beta\le1.
\end{equation}
Moreover, the function $\beta\mapsto I_\R(\alpha,\beta)$ is constant on the interval $[-1,1]$ and under
the above restrictions on~$\alpha$ and $\beta$ we have
\begin{equation}
\label{min-bet}
 I_\R(\alpha,\beta)
 =-\frac12+\frac12\sqrt{1+4\alpha}.
\end{equation}
\end{proposition}
The proof of this proposition, that we only sketch, relies on the identity 
$p(x)+\beta p'(x)=\frac{1+\beta}{2}e^{x}{\bf 1_{\R^-}} +\frac{1-\beta}{2}e^{-x}{\bf 1}_{\R^+}$, implying that
\[
I_\R(\alpha,\beta)=
\inf\biggl\{ \int_0^\infty e^{-x}\bigl(\alpha u^2+ u_x^2\bigr)\dd x \colon u\in H^1(\R^+), \;u(0)=1\biggr\}.
\]
For $\alpha>0$, one easily find solving the associate Euler--Lagrange equation 
the minimizer $\bar u\in H^1(\R^+)$ given by 
$\bar u(x)=\exp\bigl(x(\frac12-\frac12\sqrt{1+4\alpha})\bigr)$ and the minimum by formula~\eqref{min-bet}.

For $-\frac14\le \alpha<0$, $\bar u$ does not belong to $H^1(\R^+)$. However, the inequality 
$I_\R(\alpha,\beta)\ge-\frac12+\frac12\sqrt{1+4\alpha}$ does hold, as proved in~\cite{BraCMP}.
To see that equality~\eqref{min-bet} remains true in this case, we can construct a minimizing
sequence taking $u_\ell=\bar u$ on $[0,\ell]$ and $u_\ell(x)=C_\ell \exp(-\sqrt{-\alpha} x)$ for $x\ge\ell$,
with $C_\ell$ chosen in a such way that $u_\ell$ is continuous.
In this way,  $\int_\ell^{+\infty}e^{-x}\bigl(\alpha u^2+ u_x^2\bigr)\dd x =0$
and letting $\ell\to\infty$ we find 
$\int_0^\ell e^{-x}\bigl(\alpha u^2+ u_x^2\bigr)\dd x \to-\frac12+\frac12\sqrt{1+4\alpha}$.

For $\alpha<-1/4$, we choose $b>1/2$ such that $\alpha+b^2<0$.
Considering now continuous functions on $\R^+$
of the form $u_\ell(x)=e^{bx}$ on $[0,\ell]$ and $C_\ell \exp(-\sqrt{-\alpha} x)$ for $x\ge\ell$
and letting $\ell\to\infty$ we easily see that $I(\alpha,\beta)=-\infty$ in this case.

Notice that the condition $\alpha\ge-1/4$ gives the restriction $\gamma\le4$.
On the other hand, the formula
$\beta_\gamma=\inf\{\beta\ge0\colon \beta^2+ I_\R(\alpha,\beta)-\alpha\ge0\}$ immediately  gives
$\beta_\gamma=+\infty$ if $\gamma<1$. We recover in this way that we must restrict ourselves to 
$\gamma\in[1,4]$ and that $\beta_\gamma$ is given by formula~\eqref{begar} in the non-periodic case.

\section{Appendix: numerical analysis}
\label{sec:app}

In this appendix we briefly revisit Theorem~\ref{maintheo} using a numerical approach.
Rather than obtaining estimates for $\beta_\gamma$ by making use of some exact formulas,  as we did in Theorem~\ref{theoquant}, we can evaluate
$\beta_\gamma$ numerically. 
In order to achieve this, we need first to approximate $I(\alpha,\beta)$.
This can be done by approaching the solutions of the boundary value problem~\eqref{ELa1}, 
for any fixed $\alpha$ and $\beta$, with $0\le \beta<\frac{e+1}{e-1}$. This 
can be done with arbitrary high precision applying {\it e.g.\/} a finite difference technique with Richardson extrapolation.
In order to get an error smaller than $0.001$ we needed a large number of grid points (a few thousands), especially  when $\beta$ is getting close to the critical value $\frac{e+1}{e-1}$.

The following table provides the values of~$\beta_\gamma$ computed numerically, corresponding to the values of~$\gamma$ listed in~\cite{Dai-Huo} and associated with known hyper-elastic materials. 
Only in one case ($\gamma=-0.539$) our theorems are not applicable. 
\medskip
\begin{center}
\begin{tabular}{|l|cccccccccc|}
\hline
$\phantom{\Bigl|}\gamma\phantom{\bigl|}$ 		&-29.476 & -4.891 	& -2.571	& -1.646 	& -0.539 	&  1.010 &  1.236 &1.700  & 2.668  & 3.417 \\
\hline
$\phantom{\Bigl|}\beta_\gamma\phantom{\bigl|}$	&0.326 	& 0..492 & 0.684 	& 0.933 	& n.a.	 &  0.507 & 0.375 &.0.207 & 0.035 & 0.035\\
\hline
\end{tabular}
\end{center}
\medskip

\begin{figure}
\begin{minipage}[c]{0.55\linewidth}
\vskip-6.5cm
\includegraphics[width=12cm,height=12cm]{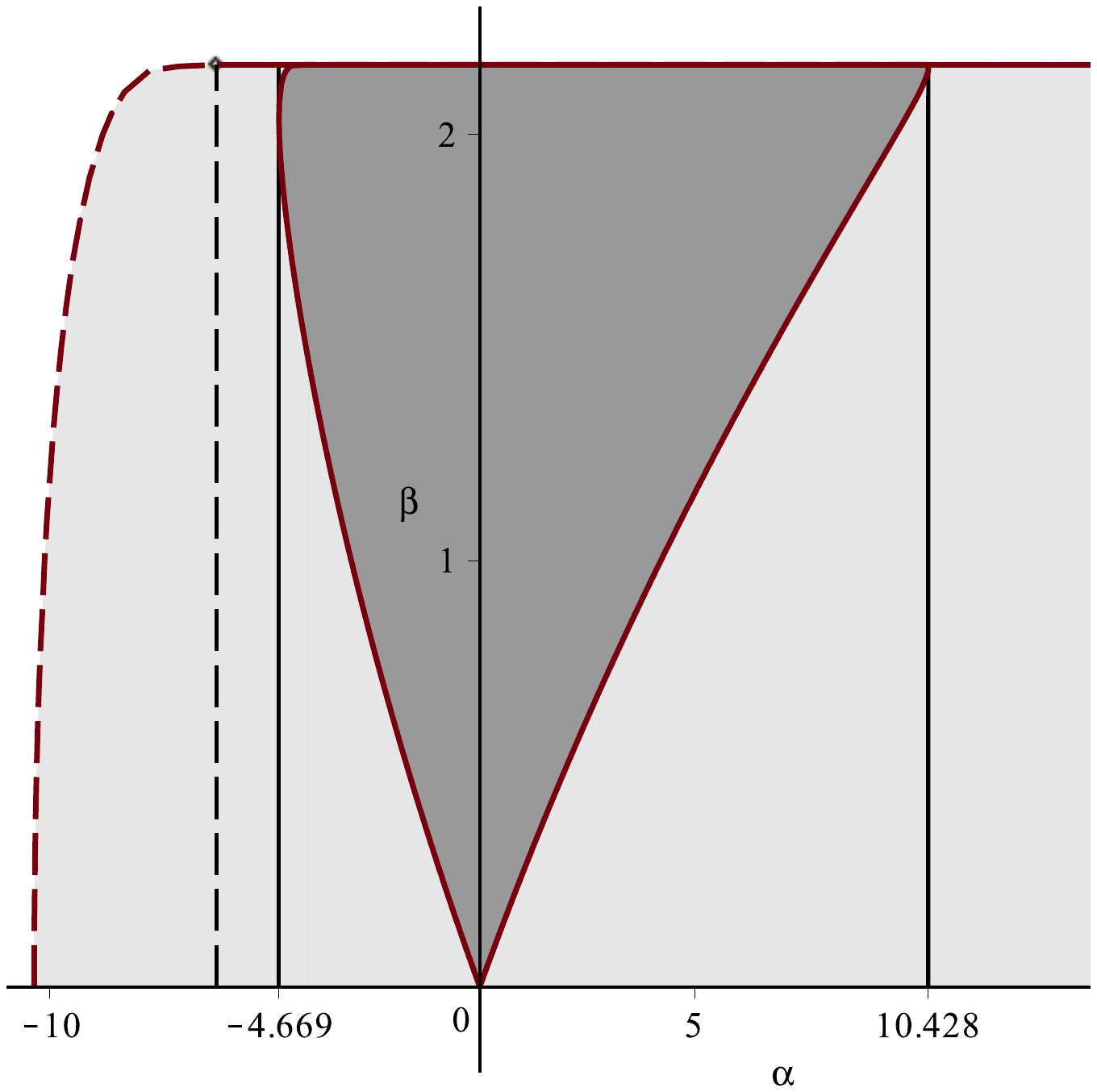}{\centering}
\vskip-5.2cm
\caption{Set of points $(\alpha,\beta)$ such that $\beta\ge0$ and 
$\beta^2+I(\alpha,\beta)-\alpha\ge0$ (dark gray region). Right: two zooms on the upper-left and upper-right corners.}
\end{minipage}
\hfill
\begin{minipage}[c]{.44\linewidth}
\qquad
\includegraphics[width=11cm,height=12cm]{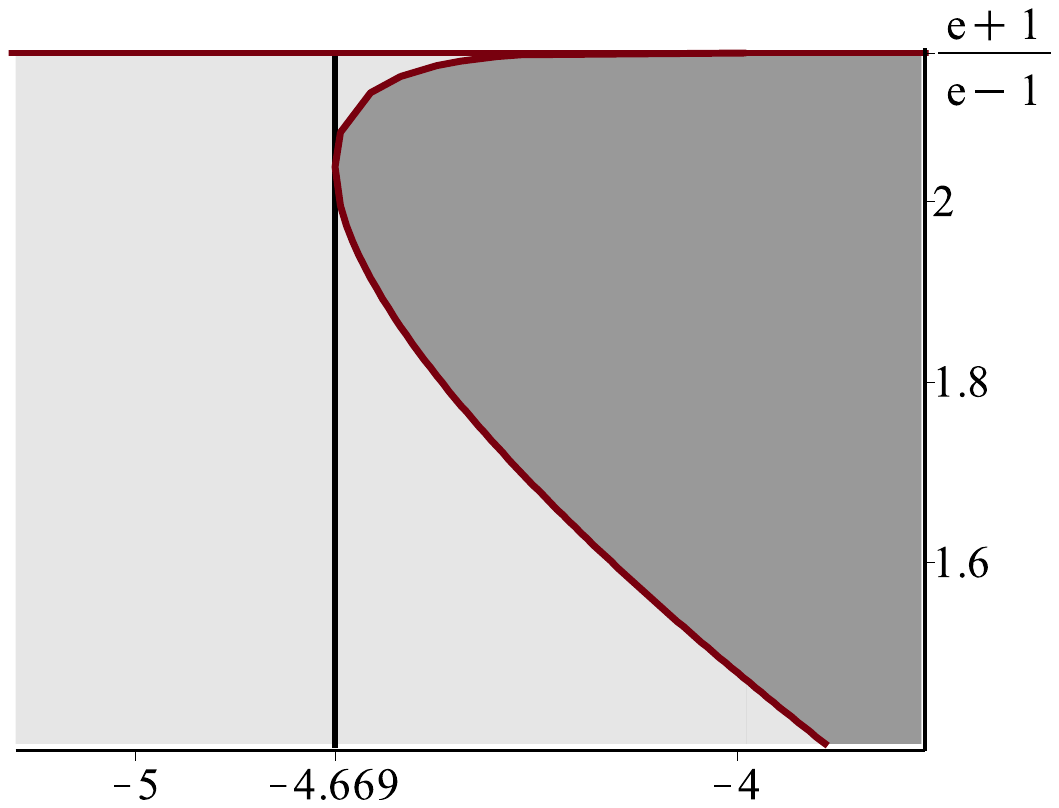}{\centering}
\vskip-7.5cm
\includegraphics[width=11cm,height=12cm]{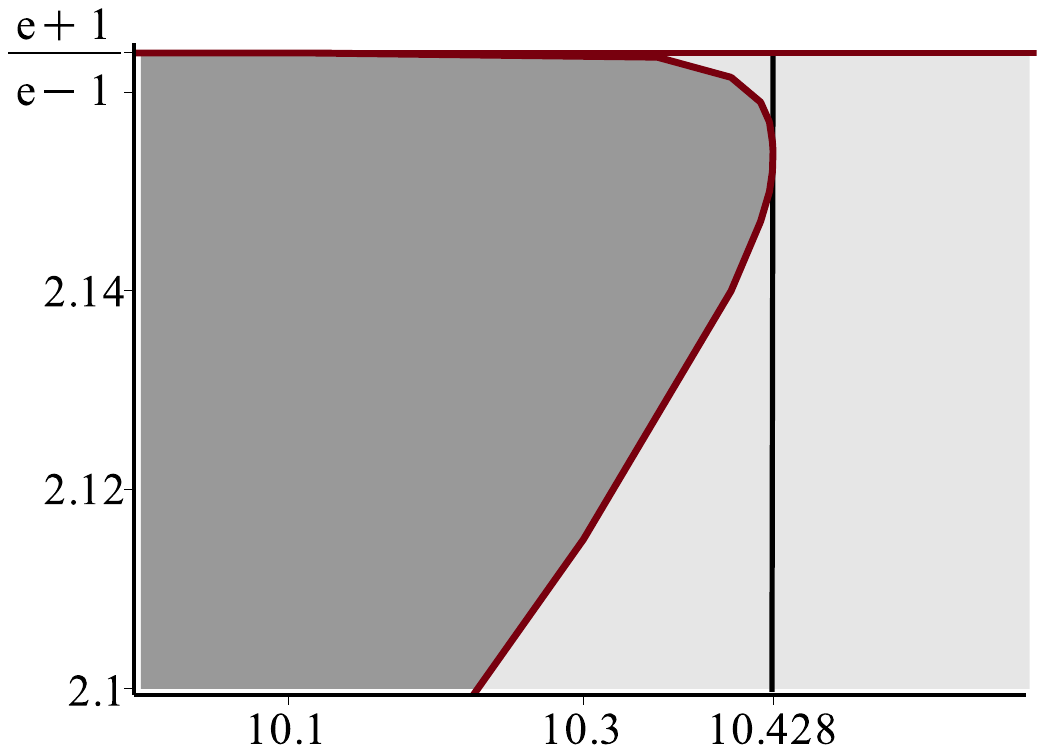}{\centering}
\end{minipage}
\vskip-7cm
\label{feuille}
\end{figure}

Picture~\ref{feuille} illustrates the set of points~$(\alpha,\beta)$ such that 
$\beta\ge0$ and $\beta^2+I(\alpha,\beta)-\alpha\ge0$. 
It turns out that this set is nonempty (approximatively) for
$\alpha\in [-4.669\,, 10.428]$. 
Recalling the definition of~$\beta_\gamma$ given in~\eqref{def2:beta}-\eqref{gal}, we see that
Theorem~\ref{maintheo} can be applied for 
$\gamma\in (-\infty,-0.817\ldots]\cup[0.262\ldots,+\infty)$,
which is a range 
slightly larger than the range  $\gamma\in(-\infty,\gamma_1^-]\cup[\gamma_1^+,\infty)$, predicted by
Theorem~\ref{theoquant}.
Relying in this numerical analysis, we claim that Corollary~\ref{cnge} remains true in such a larger range.

\begin{figure}
\includegraphics[width=15cm,height=9cm]{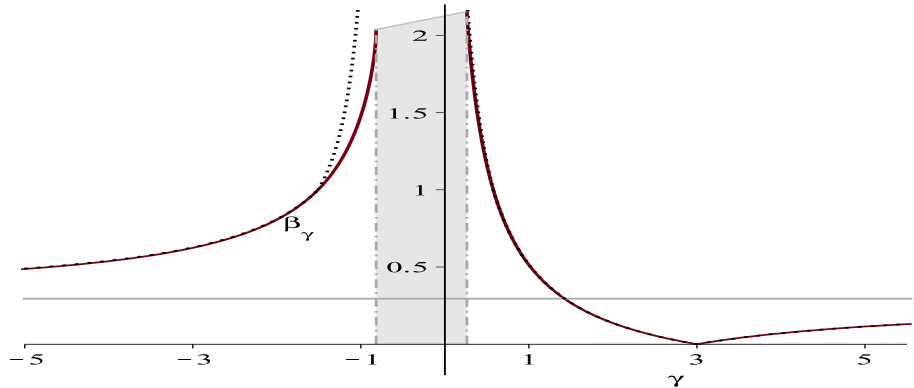}{\centering}
\vskip-3.5cm
\caption{The plot of $\beta_\gamma$ in the interval 
$\gamma\in (-\infty,-0.817]\cup[0.262,+\infty)$ (continuous line) and the upper estimate of $\beta_\gamma$ provided by Theorem~\ref{theoquant} (dotted line).}
\label{courbe-beta-gamma}
\end{figure}

Figure~\ref{courbe-beta-gamma} provides the plot of the the function $\gamma\mapsto\beta_\gamma$ as computed numerically. The curve is very close to that in Figure~\ref{upbound-betag} (reproduced here through a dotted line) and the two curves almost overlap,
in particular for $\gamma>\gamma_1^+$ or $\gamma<\gamma_2^+$ (only when 
$\gamma\in [\gamma_2^-,\gamma_1^-]$ the difference between the two curves becomes visible).

%

%


\section{Acknowledgements}
The authors are grateful to the referee for his useful suggestions and for pointing out a few relevant geometrical interpretations of the model studied in the present paper.

\end{document}